\documentclass[11pt,a4paper,reqno]{amsart}
\usepackage{amsaddr}
\usepackage{amsmath,amsfonts,amssymb,amsthm,mathrsfs}
\usepackage[utf8]{inputenc}
\usepackage[T1]{fontenc}
\usepackage{lmodern}
\usepackage{latexsym}
\usepackage{cancel}
\usepackage{microtype}
\usepackage{caption}
\usepackage{MnSymbol}
\usepackage{xcolor}
\usepackage{enumerate}
\usepackage[normalem]{ulem}
\usepackage{bbm}
\usepackage{geometry}
\usepackage{marginnote}
\usepackage{mathrsfs}

\usepackage[colorlinks=true,linkcolor=blue,citecolor=magenta]{hyperref}

\usepackage{graphics,tikz}

\newcommand{\AAA}{{\mathcal A}}

\newcommand{\FF}{{\mathcal  F}}

\newcommand{\HH}{{\mathcal  H}}
\newcommand{\TT}{{\mathcal  T}}

\newcommand{\VV}{{\mathcal  V}}

\newcommand{\ve}{\varepsilon}

\def\XXint#1#2#3{{\setbox0=\hbox{$#1{#2#3}{\int}$ }
\vcenter{\hbox{$#2#3$ }}\kern-.6\wd0}}

\newcommand{\esssup}{\mathop{\mathrm{ess\,sup}}}

\newtheorem{theorem}{\bf Theorem}[section]
\newtheorem{proposition}[theorem]{\bf Proposition}
\newtheorem{lemma}[theorem]{\bf Lemma}
\newtheorem{corollary}[theorem]{\bf Corollary}

\theoremstyle{definition}
\newtheorem{definition}[theorem]{Definition}

\newtheorem{remark}[theorem]{Remark}

\numberwithin{equation}{section}

\begin{document}

\title[Mokobodzki's  intervals: an approach  to 
 Dynkin games]{Mokobodzki's  intervals: an approach  to 
 Dynkin games when value process is not a semimartingale }

\maketitle
\begin{center}
 \normalsize
  TOMASZ KLIMSIAK\footnote{e-mail: {\tt tomas@mat.umk.pl}}\textsuperscript{1,2} \,\,\,
  MAURYCY RZYMOWSKI\footnote{e-mail: {\tt maurycyrzymowski@mat.umk.pl}}\textsuperscript{2}   \par \bigskip
  \textsuperscript{\tiny 1} {\tiny Institute of Mathematics, Polish Academy of Sciences,\\
 \'{S}niadeckich 8,   00-656 Warsaw, Poland} \par \medskip
 
  \textsuperscript{\tiny 2} {\tiny Faculty of
Mathematics and Computer Science, Nicolaus Copernicus University,\\
Chopina 12/18, 87-100 Toru\'n, Poland }\par
\end{center}

\begin{abstract}
We study  Dynkin games governed by a nonlinear $\mathbb E^f$-expectation  on a finite interval $[0,T]$, with 
payoff c\`adl\`ag processes $L,U$  of   class (D)  which are not imposed to  satisfy (weak) Mokobodzki's condition
 -- the existence of a c\`adl\`ag semimartingale between the barriers.
 For that purpose we introduce the notion of Mokobodzki's stochastic intervals $\mathscr M(\theta)$ 
(roughly speaking, maximal stochastic interval on which Mokobodzki's condition is satisfied when starting from the
stopping time $\theta$)
and  the notion of reflected BSDEs without  Mokobodzki's condition (this is a generalization and modification  
of the notion introduced by Hamad\'ene and Hassani (2005)). We prove an existence and uniqueness result
for  RBSDEs with driver $f$ that is non-increasing with respect to the value variable (no restrictions on the growth) 
and Lipschitz continuous with respect to the control variable,   and with  data in $L^1$ spaces.
Next, by using RBSDEs, we 
show numerous  results on Dynkin games: existence of the value process, saddle points, and convergence of the penalty scheme.
We also show that the game is not played beyond $\mathscr M(\theta)$, when starting from $\theta$.

\end{abstract}
\maketitle
\footnotetext{{\em Mathematics Subject Classification:}
Primary  ; Secondary .}


\section{Introduction}
\label{roz1}

Let $(\Omega,\FF,\mathbb P)$ be a complete probability space,    $(B_t)_{t\ge 0}$ denote 
 a Brownian motion on $(\Omega,\FF,\mathbb P)$ and $\mathbb F:=(\FF_t)_{t\ge 0}$  stand for the  augmented  filtration generated by 
$B$, i.e. $\FF_t:= \sigma(B_s: s\in [0,t])\vee \mathcal N$, where $\mathcal N\subset \FF$ is the family of all $\mathbb P$-null sets.
We fix a terminal time  $T>0$.  Let $\TT$ denote the set of all $\mathbb F$-stopping times $\tau$ such that $\tau\le T$
and for given $\tau\in\TT$, $\TT_\tau\subset \TT$ consists of $\sigma$ satisfying  $\sigma\ge\tau$.

\subsection{The concept of Mokobodzki's intervals}
Suppose that  $L,U$ are  c\`adl\`ag $\mathbb F$-adapted processes of class (D)  such that 
\[
L_t\le U_t,\quad t\in [0,T].
\]
In the present paper we introduce the following notions:  for given $\tau\in\TT$
we let 
\begin{equation}
\label{eq1.min1}
\Sigma_{\tau}(L,U):=\left\{\sigma\in \mathcal T_\tau: \exists \text{ c\`adl\`ag semimartingale } X \text{ such that } 
L_t\le X_t\le U_t,\, t\in [\tau,\sigma]\right\}.
\end{equation}
Next, for any $\tau\in\TT$, we let
\begin{equation}
\label{eq1.min2}
\mathring\tau:=\esssup\{\sigma: \sigma\in\Sigma_{\tau}(L,U) \}.
\end{equation}
Notice that the set $\Sigma_{\tau}(L,U)$ consists of stopping times $\sigma\ge \tau$ such that
 Mokobodzki's condition holds  for $L, U$  on $[\tau,\sigma]$. 
As to the point $\mathring\tau$  it can  be thought of as a threshold time 
after which Mokobodzki's condition for $L,U$, when starting from the time $\tau$,  ceases to be satisfied.
We prove that $\mathring\tau$ is a stopping time. 
However, the most compelling question is about the  structure of $\mathring\tau$ and 
the behavior of barriers $L,U$ at $\mathring\tau$. Clearly, we cannot expect, in general, that Mokobodzki's condition
still holds on $[\tau,\mathring\tau]$.

\subsection{Main results I}
We prove (see Theorem \ref{prop7.3m}) that there exists a (unique) set $\mathscr C_\tau\in\FF_\tau$ such that
\begin{equation}
\label{eq1.min}
\mathscr M(\tau):= [[\tau,\mathring\tau[[\,\cup\,[[\mathring\tau_{|\mathscr C_\tau}]]\subset \Omega\times [0,T]
\end{equation}
is a maximal stochastic interval on which Mokobodzki's condition is satisfied, i.e.
\begin{itemize}
\item $[[\tau,\sigma]]\subset \mathscr M(\tau)$ for any $\sigma\in \Sigma_{\tau}(L,U)$;
\item there exists a non-decreasing sequence $(\tau_k)\subset \Sigma_\tau(L,U)$ such that 
\[
\bigcup_{k\ge 1}[[\tau,\tau_k]]=\mathscr M(\tau).
\]
\end{itemize}
We call $\mathscr M(\tau)$   Mokobodzki's intervals. For those $\omega\in\Omega$ for which $\mathscr M(\tau)$
is closed, we say that Mokobodzki's condition holds  up to $\mathring\tau$.
As to  the behavior of barriers $L,U$ at $\mathring\tau$,
we prove (see Theorem \ref{prop7.3m}), among others, that
\begin{itemize}
\item at the threshold time $\mathring\tau$ the barriers "have to meet somehow":
\[
L_{\mathring\tau-}=U_{\mathring\tau-}\,\text{ or }\,L_{\mathring\tau}=U_{\mathring\tau} \quad\text{on}\quad \{\tau<\mathring\tau<T\},
\]
\item if   Mokobodzki's condition does not hold up to time $\mathring\tau$, then the distance between the barriers tends to
zero  when approaching $\mathring\tau$ from the left:
\[
 \bigcap_{\sigma\in\Sigma_\tau(L,U)} \{\sigma<\mathring\tau\}\subset \{L_{\mathring\tau-}=U_{\mathring\tau-}\},
\]
\item if the distance between the barriers,  when approaching  the threshold time $\mathring\tau$ from the left,  is bounded away from zero,
then Mokobodzki's 
condition holds up to $\mathring\tau$:
\[
 \{L_{\mathring\tau-}<U_{\mathring\tau-}\}\subset \mathscr C_\tau.
 \]
\end{itemize}
The above notion also may be applied to any $\mathbb F$-adapted c\`adl\`ag process $X$ by taking   the  pair $(L,U)=(X,X)$.
Then $\mathring\tau$  may be   thought of as a threshold time 
after which $X$, when starting from the time $\tau$,  "ceases to be a semimartingale".

\subsection{Main results II: Dynkin games with nonlinear expectation}
\label{sec.int3}

There is an extensive literature on (non-linear) Dynkin games
but the vast majority of it imposes   Mokobodzki's condition on  barriers $L,U$, 
i.e. existence of a c\`adl\`ag semimartingale $X$ between   
the barriers. 
At this point, the nomenclature for Mokobodzki's condition should be clarified. 
What we call here {\em Mokobodzki's condition }   is  often called in the literature  {\em weak Mokobodzki's condition}, 
and   {\em Mokobodzki's condition} is frequently understood in the literature as {\em weak Mokobodzki's condition}
plus some integrability  of the  finite variation and martingale  components
from the Doob--Meyer decomposition of $X$.   In the present paper 
we do not distinguish between these two concepts and we always understand Mokobodzki's condition as the existence of
a c\`adl\`ag semimartingale between the barriers. 
This comment  is to alert the reader that 
there are many papers  that study Dynkin games  "without  Mokobodzki's condition" 
but in a sense that there is a semimartingale between the barriers anyway 
(see e.g.  \cite{Hassairi,HH1,HO2}).   
The reason why Mokobodzki's condition is very often  
assumed  in the literature devoted to  Dynkin games 
is that it is equivalent to the value function being a semimartingale (see $V^f$ defined below),
and this in turn allows the use of the very fruitful so-called {\em  martingale method}.

The main motivation  behind the introducing of the  concept  of 
Mokobodzki's intervals is to 
show that suitably modified martingale method is also applicable to Dynkin games with barriers that do not satisfy 
Mokobodzki's condition or, equivalently, with the value process $V^f$  that  is not a semimartingale.

Consider  $\FF_T$-measurable random variable $\xi$, and  
   payoff  $J(\cdot,\cdot)$ that for given stopping times $\tau,\sigma\in \mathcal T$ admits the form
 \begin{equation}
\label{eq.intr2}
 J(\tau,\sigma)=L_{\tau}\mathbf{1}_{\{\tau \le\sigma,\tau<T\}}+U_{\sigma}\mathbf{1}_{\{\sigma<\tau\}}+\xi\mathbf{1}_{\{\tau=\sigma=T\}}.
 \end{equation}
In the game under consideration player I  chooses stopping time $\tau$,  player II chooses stopping time $\sigma$,
and $J(\tau,\sigma)$ represents the amount paid by  player II to player I.
The basic problems of the theory are   to study the existence and structure of  (dynamical) saddle points, i.e. 
 families 
$\{(\tau^*_\theta,\sigma^*_\theta),\, \theta\in\TT\}$ of stopping times such that for any $\theta\in \TT$,
\[
\mathbb{E}^f_{\theta,\tau\wedge\sigma^*_\theta}J(\tau,\sigma^*_\theta)\le \mathbb{E}^f_{\theta,\tau^*_\theta\wedge\sigma^*_\theta}J(\tau^*_\theta,\sigma^*_\theta)
\le \mathbb{E}^f_{\theta,\tau^*_\theta\wedge\sigma}J(\tau^*_\theta,\sigma),\quad \tau,\sigma\in\TT_\theta,
\]
where $\mathbb E^f$ is the nonlinear expectation, and to study the existence, regularity and structure  
of the value process $V^f$ that aggregates, if this is the case,  the family 
of random variables $V^f(\theta):=\overline{V}^f(\theta)=\underline{V}^f(\theta)$ (if the last equation holds we say that the game has value), where 
\begin{equation}
\label{eq.intr2f}
\begin{split}
\overline{V}^f(\theta)&:=\mathop{\mathrm{ess\,inf}}_{ \sigma \in\TT_\theta}\mathop{\mathrm{ess\,sup}}_{ \tau\in \TT_\theta}\mathbb{E}^f_{\theta,\tau\wedge\sigma}J(\tau,\sigma),\quad \underline{V}^f(\theta):=
\mathop{\mathrm{ess\,sup}}_{\tau\in\TT_\theta}\mathop{\mathrm{ess\,inf}}_{\sigma\in\TT_\theta}\mathbb{E}^f_{\theta,\tau\wedge\sigma}J(\tau,\sigma).
\end{split}
\end{equation}
If there exists process $V^f$
we say that the value is aggregable.  
Recall that, by the definition, for a given  progressively measurable  function
\[
f:\Omega\times [0,T]\times \mathbb R\times\mathbb R^d\to \mathbb R,
\]
stopping times $\tau\le\sigma$ and $\FF_\sigma$-measurable random variable $X$,  
\[
\mathbb E^f_{\tau,\sigma}(X):=Y_\tau,
\]
where $(Y,Z)$ is a unique solution  to backward SDE
\[
Y_t=X+\int_t^{\sigma}f(s,Y_s,Z_s)\,ds-\int_t^{\sigma}Z_s,\,dB_s,\quad t\in [0,\sigma].
\]
Observe that   if $f\equiv 0$, then $\mathbb E^f_{\tau,\sigma}(X)=\mathbb E(X|\FF_\tau)$.
We note at this point that the  additional advantage of the present paper  are the very weak assumptions on the data
(this is the benefit of the method used here).
We assume that 
\begin{itemize}
\item $\xi$ is  integrable, $L,U$ are of class (D),
\item $f$ is continuous and non-increasing with respect to the  $y$-variable,
and Lipschitz continuous with respect to $z$ variable,
\item $f(\cdot,0,0)$ is integrable, $f$ admits strictly sub-linear growth with respect to $z$-variable and 
no growth restrictions on $y$-variable (strict sub-linearity may be dropped in case $L,U\in\mathcal S^p(0,T)$ for some $p>1$).  
\end{itemize}

Based on the concept of Mokobodzki's intervals, we extend the martingale method  and prove the following results:
\begin{enumerate}
\item Dynkin game under nonlinear expectation $\mathbb E^f$ has the  value, i.e. 
\[
V^f(\theta):=\overline V^f(\theta)=\underline V^f(\theta)
\]
for any $\theta\in\TT$,
where $\overline V^f(\theta), \underline V^f(\theta)$ are defined by \eqref{eq.intr2f},
\item  the family $\{V^f(\theta),\, \theta\in\TT\}$ is aggregable and process $(V^f_t)$ that aggregates it  is a semimartingale on $\mathscr M(\theta)$ for any $\theta\in\mathcal T$, 
\item under the condition
\begin{equation}
\label{eq.intr.cg}
\Delta L_t\ge 0,\quad \Delta U_t\le 0,\quad\quad  t\in (0,T]
\end{equation}
the pair $(\tau^*_\theta,\sigma^*_\theta)$ is a saddle point for  the game \eqref{eq.intr2f}, where
\[
\tau^*_{\theta}:=\inf\{t\ge\theta,\,Y_t=L_t\}\wedge T;\quad\sigma^*_{\theta}:=\inf\{t\ge\theta,\,Y_t=U_t\}\wedge T,
\]
furthermore  $V^f$ is an $\mathbb E^f$-submartingale on $[\theta,\tau^*_\theta]$ and an $\mathbb E^f$-supermartingale 
on $[\theta,\sigma^*_\theta]$.
\end{enumerate}
As a result we show that
\[
\begin{split}
V^f(\theta)&=\mathop{\mathrm{ess\,inf}}_{\theta\le \sigma \le \mathring\theta}\mathop{\mathrm{ess\,sup}}_{\theta\le \tau\le\mathring\theta}\mathbb{E}^f_{\theta,\tau\wedge\sigma}J(\tau,\sigma)=
\mathop{\mathrm{ess\,sup}}_{\theta\le \tau\le \mathring\theta}\mathop{\mathrm{ess\,inf}}_{\theta\le \sigma\le\mathring\theta}\mathbb{E}^f_{\theta,\tau\wedge\sigma}J(\tau,\sigma)\\&
=
\mathop{\mathrm{ess\,inf}}_{\sigma\in  \Sigma_{\theta}(L,U)}\mathop{\mathrm{ess\,sup}}_{\tau\in \Sigma_{\theta}(L,U)}\mathbb{E}^f_{\theta,\tau\wedge\sigma}J(\tau,\sigma)=
\mathop{\mathrm{ess\,sup}}_{\tau\in  \Sigma_{\theta}(L,U)}\mathop{\mathrm{ess\,inf}}_{\sigma\in  \Sigma_{\theta}(L,U)}\mathbb{E}^f_{\theta,\tau\wedge\sigma}J(\tau,\sigma).
\end{split}
\]
This means that both players, starting the game at $\theta$,
need for the optimal decisions  only information about    payoff processes $L^\xi,U^\xi$ up to the  time $\mathring\theta$,
where $L^\xi=\mathbf1_{\{t<T\}}L_t+\mathbf1_{\{t=T\}}\xi$, $U^\xi=\mathbf1_{\{t<T\}}U_t+\mathbf1_{\{t=T\}}\xi$.
Observe that if $\mathring\theta<T$, then the players do not need information about $\xi$ at all.
By (ii) we have that  for any $\theta\in\TT$ the value process $V^f$ is a semimartingale on 
$\mathscr M(\theta)$, i.e.
for any $\sigma \in\Sigma_{\theta}(L,U)$ 
process $[0,T]\ni t\mapsto V^f_{(t\vee \theta)\wedge \sigma}$ is a c\`adl\`ag semimartingale.
Consequently, by the Doob--Meyer decomposition, for any  $\theta\in\TT$ there exist 
c\`adl\`ag finite variation process $C^{[\theta]}$ on $\mathscr M(\theta)$, with $C^{[\theta]}_\theta=0$,
and progressively measurable process $Z^{[\theta]}$ on $\mathscr M(\theta)$ such that 
\[
V^f_t=V^f_\theta +C^{[\theta]}_t+\int_\theta^t Z_s^{[\theta]}\,dB_s,\quad t\in \mathscr M(\theta).
\]
This enables us to define another dynamical saddle point for the game \eqref{eq.intr2f}: 
 for any $\theta\in\TT$ let
\begin{equation}
\label{eq.sadph}
\hat\tau_\theta:= \inf\{t\in \mathscr M(\theta): R^{[\theta],+}_t>0\},\quad \hat\sigma_\theta:= \inf\{t\in \mathscr M(\theta): R^{[\theta],-}_t>0\},
\end{equation}
where 
\[
R^{[\theta]}_t:= -C^{[\theta]}_t-\int_\theta^tf(s,Y_s,Z^{[\theta]}_s)\,ds,\quad t\in \mathscr M(\theta).
\]
The culmination of the concepts introduced in the paper  is the following  result (see Theorem \ref{th.mal}):
whenever $(\tau_\theta,\sigma_\theta)$ is a saddle point for the game \eqref{eq.intr2f}
then 
\begin{equation}
\label{eq.max}
\tau_\theta\wedge \sigma_\theta\le \hat\tau_\theta\wedge \hat\sigma_\theta
\end{equation}
and $V^f$ is a martingale on $[\theta, \tau_\theta\wedge \sigma_\theta]$. 
The result  gives a picture of the idea of the  {\em martingale method}
 for non-semimartingale Dynkin games (i.e.  games in which the value process is not a semimartingale).   
It  means that  the  game starting at time $\theta$ is played no beyond 
the point at which the barriers "cease to satisfy the Mokobodzki condition"
or, in other words, no beyond 
the point at which the value process $V^f$ "ceases to be a semi-martingale".
The game is played out right on the Mokobodzki's interval $\mathscr M(\theta)$.
This is the essence of the extended martingale method.  Note, however,  that in general $\mathscr M(\theta)$
is not a closed stochastic interval, so we cannot use the  tempting formula: "we may assume without loss of generality that $L,U$
satisfy Mokobodzki's condition" (but it is somehow close   to it).

\subsection{Main results III: reflected BSDEs without Mokobodzki's condition}
\label{sec.int4}

 If one looks at Reflected BSDEs and their connection  to Dynkin games 
one can see that the theory, in the context of Dynkin games, embraces 
the martingale method (it just uses a different language, see e.g. the comments following \cite{KZ}, Theorem 9.1). 
However, BSDEs theory is a much more convenient tool for studying advanced models in stochastic games and optimal control theory. 
In recent years, the need to analyze  generalized variants of Dynkin games and related control problems
has become strongly apparent.  
Therefore, in the present paper, we shall  develop the martingale method by 
 using  the   BSDEs framework, which at the same time allows us to  study  general Dynkin game  \eqref{eq.intr2f}.

Recall that  by the classical  definition  (see e.g. \cite{CK}) 
a solution to RBSDE on $[0,T]$ with terminal condition $\xi$, driver $f$ and barriers $L,U$
is a triple $(Y,Z,R)$ of $\mathbb F$-progressively measurable processes --- $Y, R$ are c\`adl\`ag and $R$ is of finite variation, with $R_0=0$ ---
 such that
 \[
 Y_t=\xi+\int_t^T f(s,Y_s,Z_s)\,ds +(R_T-R_t)-\int_t^T Z_s\,dB_s,\quad\quad L_t\le Y_t\le U_t, \quad t\in [0,T].
 \] 
Furthermore, to guarantee  uniqueness and   connection  of $(Y,Z,R)$
to Dynkin games process $R$ is imposed to satisfy the Skorokhod condition
\[
\int_0^T(Y_{s-}-L_{s-})\,dR^+_s=\int_0^T(U_{s-}-Y_{s-})\,dR^-_s=0.
\]
We call such solutions {\em classical} or {\em semimartingale} or $\mathscr S${\em-solution}.
Observe that by the definition $Y$ is a semimartingale and it lies between the barriers.
Thus, without Mokobodzki's condition on the  barriers $L,U$ classical definition cannot be applied.
We see that in order to study RBSDEs without Mokobodzki's condition a new definition of a solution is required.
Let us also make  one more remark. Observe that processes $Z,R$ are uniquely determined by
process $Y$. Indeed, any c\`adl\`ag $\mathbb F$-semimartingale $X$ admits a unique Doob--Meyer decomposition
\[
X_t=X_0+V_t+\int_0^t H_s\,dB_s,\quad t\in [0,T]
\]
(we also used here representation property of $\mathbb F$), where $H$ is a progressively measurable process
and $V$ is a c\`adl\`ag finite variation process. Letting 
\[
\mathscr V(X):=V,\quad \mathscr Z(X):=H,
\]
we get that $Z=\mathscr Z(Y)$ and $R=\mathscr V(Y+\int_0^\cdot f(s,Y_s,Z_s)\,ds)$. 
Therefore, one might as well say that solely $Y$ is a solution to aforementioned RBSDE
and treat processes $Z$, $R$ as functions of $Y$.

The notion of RBSDEs without Mokobodzki's condition  has been introduced in \cite{HH1,HH2}
on the Brownian filtration and the  Poisson--Brownian filtration,  and
extended to a general filtration in \cite{K:SPA} but for drivers depending only on $y$-variable.
Our goal is to modify and simplify the notion   in order to 
get in a relatively simple way the existence
result for the mentioned class of RBSDEs and then, as a  corollary,   to get 
a number of results for Dynkin games \eqref{eq.intr2f}. 
We propose the following definition.

{\em We say that a c\`adl\`ag $\mathbb F$-adapted process $Y$ is a solution to RBSDE on $[0,T]$ with terminal condition $\xi$, driver $f$ and barriers $L,U$
if $Y_T=\xi$,  and for any $\tau\in\TT$ and  $\sigma\in \Sigma_{\tau}(L,U)$, $Y$ is a classical
solution to RBSDE on $[\tau,\sigma]$ with terminal condition $Y_\sigma$, driver $f$ and barriers $L,U$.}

A similar approach has been applied  in \cite{HH1} but with a subfamily
\[
\tilde\Sigma^Y_\tau\subset \Sigma_{\tau}(L,U)
\]
that depends on the solution (see \cite{HH1}, Proposition 3.6). This makes so formulated definition very cumbersome 
in applications, especially 
when passing to the limit with a sequence of solutions $(Y^n)$ since then we
do not have a fixed  interval $[\tau,\sigma]$ (since $\sigma\in \tilde\Sigma^{Y^n}_\tau$)
on which underlying RBSDEs can be treated  as the classical one. 
In \cite{K:SPA} we modified this notion by introducing  
a family  $\bar \Sigma_\tau(L,U)$ independent of $Y$ (see  Appendix \ref{rozB}), which significantly simplified 
the proof techniques  based on it,  but still we had, in general, 
\[
\tilde\Sigma^Y_\tau\subsetneq \bar \Sigma_\tau(L,U) \subsetneq \Sigma_{\tau}(L,U).
\]   
In the present paper,  we take   in the definition of a solution to RBSDE the entire  family $\Sigma_{\tau}(L,U)$, i.e.
we solve  RBSDEs locally in the classical way where  possible. 
This is a  natural approach, and, in our opinion, 
it is the most  clear definition  
of a solution to  RBSDE without Mokobodzki's condition 
and at the same time very convenient to use.
However, the main impetus for extending the previous definitions 
(especially the last one, where $\bar \Sigma_\tau(L,U)$ was, after all, independent of the solution) were
 the applications of RBSDEs without Mokobodzki's condition
to  Dynkin games. It appears that in  general the  saddle point $(\hat \tau_\theta,\hat \sigma_\theta)$
defined in \eqref{eq.sadph} satisfies
\[
\hat \tau_\theta,\hat \sigma_\theta\in   \Sigma_\theta(L,U)\quad\text{but}\quad  \hat \tau_\theta,\hat \sigma_\theta\notin  \bar \Sigma_\theta(L,U).
\]
In other words, in the previous definitions, the intervals on which RBSDE was considered 
as the classical ones were too short to include  all possible optimal strategies.

We prove (see Theorem \ref{9wrzesnia1}) that under aforementioned assumptions on the data  there exists a unique solution $Y$
of class (D) to RBSDE on $[0,T]$ with terminal condition $\xi$, driver $f$ and barriers $L,U$ and
\[
Y_\theta=V^f(\theta),\quad \theta\in\TT.
\]
Moreover,
\begin{equation}
\label{eq1.pen}
Y^n_t\to Y_t,\quad t\in [0,T],
\end{equation}
where $(Y^n,Z^n)$ is a unique solution to BSDE (penalty scheme)
\[
Y^n_t=\xi+\int_t^Tf(s,Y^n_s,Z^n_s)\,ds+n\int_t^T(Y^n_s-L_s)^-\,ds-n\int_t^T(Y^n_s-U_s)^+\,ds-\int_t^TZ^n_s\,dB_s.
\]
Furthermore, the limit is uniform, more precisely,
\[
\sup_{0\le t\le T}|Y^n_t-Y_t|\to 0,\quad P\text{-a.s.}
\]
provided that  \eqref{eq.intr.cg} holds.
We also provide  a comparison result for the solutions of studied RBSDEs (Theorem \ref{th.cop1}).

 \subsection{Related literature}
The main motivation for introducing the above notion stems from  study of 
Dynkin games  introduced  by Dynkin and Yushkevich \cite{DY} (1968) and later studied
by many  authors including Krylov \cite{Krylov} (1971), Neveu \cite{Neveu}(1975), Bensoussan
and Friedman \cite{BF}(1974), Bismut \cite{Bismut,Bismut1}(1977/1979), Stettner \cite{Stettner}(1984), Zabczyk \cite{Zabczyk}(1984),
Morimoto \cite{Morimoto} (1984),  Alario-Nazaret, Lepeltier and Marchal \cite{ALM}(1983), 
Lepeltier and Maingueneau \cite{LM} (1984),   Cvitani\'c and   Karatzas \cite{CK} (1996),  Touzi and  Vieille  (2002) \cite{TV}, 
Hamad\'ene and  Hassani  \cite{HH1,HH2}  (2005/2006), Ekstr\"om and  Peskir \cite{EP} (2008), Buckdahn and Li \cite{BL} (2009), Dumitrescu, Quenez and Sulem \cite{DQS}(2016), Bayraktar and Yao \cite{BY} (2017) Grigorova, Imkeller, Ouknine,  Quenez, \cite{GIOQ}  (2018) and others.

The martingale method, introduced by Snell \cite{Snell} (1952) for the optimal stopping problem,
and later developed for Dynkin games and  optimal stochastic control
(see Neveu \cite{Neveu} (1975), El Karoui \cite{EK} (1981),
 Rishel \cite{Rishel} (1970),  Davis and Varaiya \cite{DV}
(1973), Davis \cite{Davis}(1973, 1979) and Elliott \cite{Elliott}(1977, 1982), 
Morimoto \cite{Morimoto},  Karatzas and Zamifrescu \cite{KZ}, Bayraktar and Yao \cite{BY}) 
is a straightforward and probabilistic approach to optimal and control problems  that enables a  unified   in-depth study of 
Markovian and non-Markovian models. 
The essence of the method  is to  formulate Bellman's  principle of optimality  as a supermartingale inequality
 and then obtain  conditions for optimality by means of the martingale and finite variation part of $V^f$;  
 thus it requires  that $V^f$ be a semimartingale (this is equivalent to Mokobodzki's condition).  
 This seems to be  the most general form of
dynamic programming that applies  to a very general class of problems.


Using the introduced notion \eqref{eq1.min1}--\eqref{eq1.min}  we show that in fact martingale method can be applied
even if Mokobodzki's condition is not satisfied.
In order to do this, we  utilize the theory of Backward Stochastic Differential Equations (BSDEs)
introduced in  1990 by Pardoux and Peng \cite{PP} and then extended to reflected BSDEs by 
 El Karoui et al. \cite{EKPPQ} and Cvitanic and Karatzas \cite{CK} (1996/1997).  Reflected BSDEs without Mokobodzki's 
 condition imposed on the barriers has  been introduced by Hamad\'ene and Hassani in \cite{HH1}
and further  studied only in  few papers (see \cite{HH2,HHO,K:SPA}).

The only papers we are aware of   on   non-Markovian continuous time Dynkin games (on stopping times)  with
 value process that is not a semimartingale    are \cite{HH2,HZ,K:SPA,LM,Stettner}. Only the third one  was concerned 
 with non-linear expectation $\mathbb E^f$ but  with $f$ independent of the control variable $Z$.
 In \cite{HH2,HZ,LM,Stettner} the authors considered linear expectation.  
 As to the barriers, in \cite{LM,Stettner}   $L,U$ are supposed  to be bounded,
 in \cite{HH2}    $L,U \in \mathcal S^2$ but with no predictable jumps,  and in \cite{HZ} barriers are merely of class (D)
 but are imposed to satisfy \eqref{eq.intr.cg}. In all the mentioned papers  a general filtration $\mathbb F$ have been considered
 and the authors focused on the existence of the value and a saddle point.
It seems that our paper is the first one in this context (even in case $f\equiv 0$)   showing  that the pair
$(\hat\tau_\theta,\hat\sigma_\theta)$ defined in \eqref{eq.sadph} is maximal in the sense of \eqref{eq.max}.

As to the results on reflected BSDEs summarized in Section \ref{sec.int4},
we  generalize the results of \cite{HH1,HH2,HHO,HO2,K:SPA} (\cite{HO2}, Theorem 3.1, \cite{HHO}, Theorem 3.1, \cite{HH2}, Theorem 2.2,
\cite{HH1},Theorem 3.5, Proposition 3.6) in many directions.  
\begin{itemize}
\item  We show the existence result under  merely monotonicity condition on $f$ with respect to $Y$
with no restrictions on the growth (non-Lipschitz data has been considered only in \cite{K:SPA} but with $f$ independent of $Z$);
\item We merely assume that the data are integrable, i.e. $L,U$ are of class (D) and $f(\cdot,0,0), \xi$  are in $L^1$ spaces
(in \cite{HH1,HH2,HHO,HO2} barriers are in $\mathcal S^2$ and $f(\cdot,0,0), \xi$  are in $L^2$ spaces);
\item We show that $Y$ defined as the limit in \eqref{eq1.pen} (the existence of the limit is a part of the assertion)  is a unique 
c\`adl\`ag process that is  a classical solution to  
RBSDE$^{\theta,\tau}(Y_\tau\mathbf1_{\{\tau<T\}}+\xi\mathbf1_{\{\tau=T\}},f,L,U)$ for any $\tau\in \Sigma_\theta(L,U)$
(this is a complete novelty of our paper);
\end{itemize}

 \subsection{Organization of the paper}
In Section \ref{roz2} we set basic notation. In Sections \ref{roz3}, \ref{roz4} we recall the definitions
and basic results for  BSDEs and reflected BSDEs, respectively,
with a slight generalization of the comparison result in Proposition  \ref{10wrzesnia1}. 
In Section \ref{roz5} we introduce and investigate  the notion
of reflected BSDEs with barriers satisfying asymptotic Mokobodzki's condition. 
This is an interesting concept in its own right and we use it as a key tool in the ensuing sections.
The main results of the paper mentioned
in the introduction are contained in Sections \ref{roz6}--\ref{roz10}. In Section \ref{roz6} 
we introduce and study  Mokobodzki's intervals.
Section \ref{roz7} is devoted to  reflected BSDEs without Mokobodzki's condition.  
Sections \ref{roz8}--\ref{roz10} are concerned with Dynkin games.

\section{Basic notation}
\label{roz2}

Let $d\ge 1$ be a natural number. For $x\in\mathbb{R}^d$ by $|x|$ we denote 
the Euclidean norm of $x$. In what follows, unless stated otherwise,    $T>0$
is a  positive number (we shall make an exception from this rule in Section \ref{roz5}). 
Let $\mathcal A$ stands for the set of all $\mathbb F$-stopping times. 
Moreover, for $\nu,\zeta\in\mathcal{A}$, $\mathcal{A}_{\nu}^{\zeta}:=\{\tau\in\mathcal{A},\,\nu\le\tau\le\zeta\}$, 
$\mathcal{A}^{\zeta}:=\mathcal{A}_{0}^{\zeta}$.
Furthermore, we let  $\mathcal T:=\mathcal A^T$,   $\mathcal{T}_{\nu}:=\mathcal{A}_{\nu}^{T}$,  and $\mathcal T^\zeta:= \mathcal A^{\zeta\wedge T}$. 
For $\nu,\zeta\in\mathcal{A}$, we set 
\[
[[\nu,\zeta]]:=\{(\omega,t)\in\Omega\times [0,\infty): \nu(\omega)\le t\le\zeta(\omega)\},
\]
and $[[\nu]]:= [[\nu,\nu]], \, [[\nu,\zeta[[:=[[\nu,\zeta]]\setminus [[\zeta]], \,]]\nu,\zeta]]:=[[\nu,\zeta]]\setminus [[\nu]]$.
For $\tau\in\AAA$ and $A\in\FF_\tau$ we let
\[
\tau_A(\omega)=\tau(\omega),\, \omega\in A,\quad \tau_A(\omega)=\infty,\, \omega\notin A.
\]
It is well known that $\tau_A$ is a stopping time.
In what follows we adopt the  convention that any stochastic process $(X_t)_{t\ge 0}$ is  automatically be extended to $[0,\infty]$
by letting $X_\infty:= 0$.

Let $\nu,\zeta\in\mathcal{A}$, $\nu\le\zeta$, and $p\ge 1$. By $\mathcal{S}^p_{\mathbb{F}}(\nu,\zeta)$ we denote the set of  all 
real valued $\mathbb{F}$-progessively measurable  processes $Y=(Y_t)_{t\in[0,T]}$ such that
\[
||Y||_{\mathcal{S}^p(\nu,\zeta)}:=\Big(\mathbb{E}\sup_{\nu\le t\le\zeta}|Y_t|^p\Big)^{\frac{1}{p}}<\infty.
\]
$\mathcal{M}_{loc}(\nu,\zeta)$ is the space of all $\mathbb{F}$-local martingales on $[[\nu,\zeta]]$. 
Let $q\ge 1$. By $L^{p,q}_{\mathbb{F}}(\nu,\zeta)$ we denote the set of all real valued $\mathbb{F}$-progressively measurable 
processes $X=(X_t)_{t\in[0,T]}$ such that
\[
||X||_{L^{p,q}(\nu,\zeta)}:=\Bigg(\mathbb{E}\Big(\int^{\zeta}_{\nu}|X_r|^p\,dr\Big)^{\frac{q}{p}}\Bigg)^{\frac{1}{p}}<\infty.
\]
$L^p_{\mathbb{F}}(\nu,\zeta)$ is the shorthand for $L^{p,p}_{\mathbb{F}}(\nu,\zeta)$.

Let $\mathcal{G}\subset\mathcal{F}$ be a $\sigma$-field. $L^r(\mathcal{G})$ 
denotes the set of all $\mathcal{G}$-measurable random vectors $X$ such that $\mathbb{E}|X|^r<\infty$. For any $X\in L^r(\mathcal{G})$ we let $\|X\|_{L^r(\mathcal G)}:= \big(\mathbb E|X|^r\big)^{(1/r)\wedge 1}$.
For $\beta\in\AAA$ by $\mathcal{H}_{\mathbb{F}}(0,\beta)$ we denote the space of all $\mathbb{F}$-progressively measurable 
$\mathbb{R}^{d}$-valued processes $Z$ such that $\mathbb{P}(\int^{\beta}_{0}|Z_r|^2\,dr<\infty)=1$.
$\mathcal{H}^s_{\mathbb{F}}(0,\beta)$, $s>0$, is a subspace of  $\mathcal{H}_{\mathbb{F}}(0,\beta)$
consisting of  $Z$ satisfying $\mathbb{E}\Big(\int^{\beta}_{0}|Z_r|^2\,dr\Big)^{\frac{s}{2}}<\infty.$
We set
\[
||Z||_{\mathcal{H}^s(0,\beta)}:=\Bigg(\mathbb{E}\Big(\int^{\beta}_{0}|Z_r|^2\,dr\Big)^{\frac{s}{2}}\Bigg)^{\frac{1}{s}},\, s>1,\quad
|Z|_{\mathcal{H}^s(0,\beta)}:=\mathbb{E}\Big(\int^{\beta}_{0}|Z_r|^2\,dr\Big)^{\frac{s}{2}},\, s\in (0,1).
\]

We say that $\mathbb{F}$-progressively measurable process $X=(X_t)_{t\ge 0}$ 
is of class (D) on $[[\nu,\zeta]]$ if the family $\{X_{\tau},\,\tau\in\mathcal{A}_{\nu}^{\zeta}\}$ 
is uniformly integrable. By $\mathcal{D}^1_{\mathbb{F}}(\nu,\zeta)$ we denote the class of all such processes.
We endow $\mathcal D^1_{\mathbb{F}}(\nu,\zeta)$ 
with the norm
\[
||Y||_{\mathcal{D}^1(\nu,\zeta)}:=\sup_{\sigma\in\mathcal{T}_{\nu,\zeta}}\mathbb E|Y_{\sigma}|.
\]
$||\cdot||_{\mathcal{D}^1}$ is the shorthand for $||\cdot||_{\mathcal{D}^1(0,T)}$.

A sequence $(\tau_k)_{k\ge 1}\subset\mathcal{T}_{\nu,\zeta}$ is called a chain on $[[\nu,\zeta]]$ if
\[
\forall_{\omega\in\Omega}\exists_{n\in\mathbb{N}}\forall_{k\ge n}\,\tau_k(\omega)=\zeta(\omega).
\]

We let  $\mathcal{V}_{\mathbb{F}}(\nu,\zeta)$ (resp. $\mathcal{V}^+_{\mathbb{F}}(\nu,\zeta)$) 
 denote the space of all real valued $\mathbb{F}$-progressively measurable  processes $V=(V_t)_{t\ge 0}$ 
with finite variation (resp. nondecreasing) on $[[\nu,\zeta]]$ and $\mathcal{V}_{0,\mathbb{F}}(\nu,\zeta)$ 
(resp. $\mathcal{V}^+_{0,\mathbb{F}}(\nu,\zeta)$) denote  the  subspace of  $\mathcal{V}_{\mathbb{F}}(\nu,\zeta)$ 
(resp. $\mathcal{V}^+_{\mathbb{F}}(\nu,\zeta)$)  consisting of  processes $V$ such that $V_{\nu}=0$. 
$\mathcal{V}^p_{\mathbb{F}}(\nu,\zeta)$ (resp. $\mathcal{V}^{+,p}_{\mathbb{F}}(\nu,\zeta)$) is the set of all 
$V\in\mathcal{V}_{\mathbb{F}}(\nu,\zeta)$ (resp. $V\in\mathcal{V}^+_{\mathbb{F}}(\nu,\zeta)$) such that 
$\mathbb{E}|V|^p_{\nu,\zeta}<\infty$, where $|V|_{\nu,\zeta}$ denotes the total variation of $V$ on $[[\nu,\zeta]]$.

Throughout the paper all relations between random variables are supposed to hold $\mathbb{P}$-a.s. 
For processes $X^1=(X^1_t)_{t\ge 0}$ and $X^2=(X^2_t)_{t\ge 0}$ we write $X^1\le X^2$ on $[[\nu,\zeta]]$
whenever $X^1_t\le X^2_t$, $t\in[\nu,\zeta]$, $\mathbb{P}$-a.s.

We say that a c\`adl\`ag process $(X_t)$  is a semimartingale on $[[\nu,\zeta]]$ provided that 
the process $[0,\infty)\ni t\mapsto X_{(t\vee \nu)\wedge \zeta}$ is a c\`adl\`ag semimartingale.

\section{Backward SDEs}
\label{roz3}

Throughout the paper $f:\Omega\times [0,T]\times\mathbb R\times\mathbb R^d\to \mathbb R $
is assumed   $\mathbb F$-progressively measurable for fixed $(y,z)\in \mathbb R\times\mathbb R^d$. 
Let $\nu,\zeta\in\mathcal{T}$, $\nu\le\zeta$, and $\hat{\xi}\in\mathcal{F}_{\zeta}$. 

\begin{definition}
We say that a pair $(Y,Z)$ of $\mathbb{F}$-adapted processes is a solution to  backward stochastic differential equation
on the interval $[[\nu,\zeta]]$ with generator $f$ and terminal value $\hat{\xi}$ (BSDE$^{\nu,\zeta}(\hat{\xi},f)$ for short) if
\begin{enumerate}
\item[(a)] $Y$ is a continuous process  and $Z\in\mathcal{H}_{\mathbb{F}}(\nu,\zeta)$,
\item[(b)] $f(\cdot,Y,Z)$ is $\mathbb F$-progressively measurable and  $\int^{\zeta}_{\nu}|f(r,Y_r,Z_r)|\,dr<\infty$,
\item[(c)] $Y_t=\hat{\xi}+\int^{\zeta}_t f(r,Y_r,Z_r)\,dr-\int^{\zeta}_t Z_r\,dB_r$, $t\in[\nu,\zeta]$.
\end{enumerate}
\end{definition}

Let $V\in\mathcal{V}_{0,\mathbb F}(\nu,\zeta)$. 

\begin{definition}
We say that a pair $(Y,Z)$ of $\mathbb{F}$-adapted processes is a solution to  backward stochastic differential equation
on the interval $[[\nu,\zeta]]$ with right-hand side $f+dV$ and terminal value $\hat{\xi}$ (BSDE$^{\nu,\zeta}(\hat{\xi},f+dV)$ for short) if
$(Y-V,Z)$ is a solution to BSDE$^{\nu,\zeta}(\hat{\xi},f_V)$, where $f_V(t,y,z)=f(t,y+V_t,z)$.
\end{definition}

In what follows  BSDE$^{\zeta}$ is shorthand for  BSDE$^{0,\zeta}$.

We shall need the following hypotheses:
\begin{enumerate}
\item[(H1)] there is $\lambda\ge0$ such that
$|f(t,y,z)-f(t,y,z')|\le\lambda|z-z'|$ for $t\in[0,T]$, $y\in\mathbb{R}$, $z,z'\in\mathbb{R}^d$,
\item[(H2)] there is $\mu\in\mathbb{R}$ such that
$(y-y')(f(t,y,z)-f(t,y',z))\leq\mu(y-y')^2$ for $t\in[0,T]$, $y,y'\in\mathbb{R}$, $z\in\mathbb{R}^d$,
\item[(H3)] for every $(t,z)\in[0,T]\times\mathbb{R}^d$ the mapping $\mathbb{R}\ni y\rightarrow f(t,y,z)$ is continuous,
\item[(H4)] $\int^T_0 |f(r,y,0)|\,dr<\infty$ for every $y\in\mathbb{R}$,
\item[(H5)] $\xi\in L^1(\mathcal{F}_T)$, $f(\cdot,0,0)\in L^1_{\mathbb F}(0,T)$, $V\in \VV_{0,\mathbb F}^1(0,T)$,
\item[(H5$_p$)] $p>1$, $\xi\in L^p(\mathcal{F}_T)$, $f(\cdot,0,0)\in L^p_{\mathbb F}(0,T)$, 
$V\in \VV_{0,\mathbb F}^p(0,T)$,
\item[(Z)] there exists an $\mathbb{F}$-progressively measurable process $g\in L^1_{\mathbb F}(0,T)$ and $\gamma\ge 0$, $\kappa\in [0,1)$ such that
\begin{align*}
|f(t,y,z)-f(t,y,0)|\le\gamma(g_t+|y|+|z|)^\kappa,\quad t\in [0,T], \,y\in\mathbb{R},\,z\in\mathbb{R}^d.
\end{align*}
\end{enumerate}

\begin{theorem}\label{12sierpnia1p}
The following hold.
\begin{enumerate}
\item[(i)]  Assume that \textnormal{(H1)-(H4), (H5$_p$)} are in force. 
Then there exists a  solution 
$(Y,Z)\in \mathcal S^p_{\mathbb{F}}(0,T)\times\mathcal H^p_{\mathbb{F}}(0,T)$ to \textnormal{BSDE}$^T(\xi,f+dV)$. 
\item[(ii)] Assume \textnormal{(H1), (H2)}. Then there exists at most one   solution $(Y,Z)$ of 
\textnormal{BSDE}$^T(\xi,f+dV)$ such that $Y\in \mathcal S^p_{\mathbb{F}}(0,T)$.
\end{enumerate}
\end{theorem}
\begin{proof}
See \cite{bdh}, Lemma 3.1, Theorem 4.2.
\end{proof}

\begin{proposition}\label{13stycznia19}
Let  $p>1$. Assume that \textnormal{(H1), (H2), (H5)} are satisfied.
Let $(Y,Z)$ be a solution to \textnormal{BSDE}$^T(\xi,f)$ 
such that $Y\in\mathcal{S}^p_{\mathbb{F}}(0,T)$. 
Then there exists $c_p>0$, depending only on $p$, such that 
\[
\begin{split}
\mathbb E\Big[\sup_{0\le t\le T}e^{a t}|Y_t|^p+\Big(\int^T_0 e^{2a r} |Z_r|^2\,dr\Big)^{\frac{p}{2}}\Big]\le c_p\mathbb E\Big[e^{apT}|\xi|^p+\Big(\int^T_0 e^{ar}|f(r,0,0)|\,dr\Big)^p\Big],
\end{split}
\]
for any $a\ge  \mu+\frac{\lambda ^2}{1\wedge (p-1)}$.
\end{proposition}
\begin{proof}
See \cite{bdh}, Proposition 3.2.
\end{proof}

\begin{theorem}\label{12sierpnia1}
The following hold.
\begin{enumerate}
\item[(i)] Assume that \textnormal{(H1)-(H5), (Z)} are in force.
Then, there exists a  solution $(Y,Z)$ of the problem \textnormal{BSDE}$^T(\xi,f+dV)$ 
such that $Y$ is of class \textnormal{(D)} and $Z\in\mathcal{H}^s_{\mathbb{F}}(0,T)$, $s\in(0,1)$. 
\item[(ii)] Assume that \textnormal{(H1), (H2), (Z)}. Then there exists at most one   solution 
$(Y,Z)$ of the problem \textnormal{BSDE}$^T(\xi,f+dV)$ such that $Y$ is of class \textnormal{(D)}.
\end{enumerate}
\end{theorem}
\begin{proof}
See \cite{bdh}, Theorem 6.2, Theorem 6.3.
\end{proof}

The following result has been proven in \cite{KRz4}, Theorem 3.5. 

\begin{theorem}
\label{th.2}
Let $p=1$. Consider a function  $\bar f:\Omega\times [0,T]\times\mathbb R\times\mathbb R^{d}\to \mathbb R$
and $\bar\xi\in L^1(\mathcal F_T)$.
Assume that $f$ satisfies \textnormal{(H1)-(H5), (Z)}. 
Let $(Y,Z), (\bar Y,\bar Z)$ be solutions 
to \textnormal{BSDE}$^T(\xi,f)$, \textnormal{BSDE}$^T(\bar\xi,\bar f)$, respectively, 
such that $Y, \bar Y$ are of class \textnormal{(D)}. 
Suppose that $\mathbb E\int_0^T|f(r,\bar Y_r,\bar Z_r)-\bar f(r,\bar Y_r,\bar Z_r)|\,dr<\infty$. Then, for any $q\in (\kappa,1)$
there exists $\mathcal{C}>0$ depending only on $\kappa,q,||g||_{L^1(0,T)},T,$ and $\gamma$ such that for any $a\ge \mu+\frac{\lambda^2}{1\wedge(\sqrt{\frac q\kappa}-1)}$,
\[
\begin{split}
\|Y-\bar Y&\|_{\mathcal D^1(0,T)}+|Z-\bar Z|_{\mathcal H_{\mathbb{F}}^q(0,T)}\le\mathcal C\psi_3(\|\xi-\bar\xi\|_{L^1(\mathcal{F}_T)}+\||f-\bar f|(\cdot,\bar Y,\bar Z)\|_{L^1_{\mathbb F}(0,T)}),
\end{split}
\]
where $\psi_3(x)=x+x^{\kappa q(1-q)},\, x\ge 0$.
\end{theorem}

\section{Semimartingale solution to RBSDEs with two reflecting barriers}
\label{roz4}
In the whole paper, we assume that 
\begin{enumerate}
\item[(B)]   $L$ and $U$ are $\mathbb F$-adapted c\`adl\`ag processes of class (D)
on $[0,T]$,  $L\le U$, and $L_T\le\xi\le U_T$.
\end{enumerate}

Let $\nu,\zeta\in\mathcal{T}$, $\nu\le\zeta$, and $\hat{\xi}\in\mathcal{F}_{\zeta}$ such that $L_{\zeta}\le\hat{\xi}\le U_{\zeta}$. 

\begin{definition}\label{10wrzesnia3}
We say that a triple $(Y,Z,R)$ of $\mathbb{F}$-adapted processes is a semimartingale solution ($\mathscr S$-solution) to  reflected backward stochastic differential equation
on the interval $[[\nu,\zeta]]$ with generator $f$, terminal value $\hat{\xi}$, lower barrier $L$ and upper barrier $U$ (RBSDE$^{\nu,\zeta}(\hat{\xi},f,L,U)$ for short) if
\begin{enumerate}
\item[(a)] $Y$ is c\`adl\`ag and $Z\in\mathcal{H}_{\mathbb{F}}(\nu,\zeta)$,
\item[(b)] $R\in\mathcal{V}_{0,\mathbb{F}}(\nu,\zeta)$, $L_t\le Y_t\le U_t$, $t\in[\nu,\zeta]$, and
\begin{equation*}
\begin{split}
&\int^{\zeta}_{\nu}(Y_{r-}-L_{r-})\,dR^+_r+\int^{\zeta}_{\nu}(U_{r-}-Y_{r-})\,dR^-_r=0,
\end{split}
\end{equation*}
where $R=R^+-R^-$ is the Jordan decomposition of $R$,
\item[(c)] $f(\cdot,Y,Z)$ is $\mathbb F$-progressively measurable and $\int^{\zeta}_{\nu}|f(r,Y_r,Z_r)|\,dr<\infty$,
\item[(d)] $Y_t=\hat{\xi}+\int^{\zeta}_t f(r,Y_r,Z_r)\,dr+(R_\zeta-R_t)-\int^{\zeta}_t Z_r\,dB_r$, $t\in[\nu,\zeta]$.
\end{enumerate}
\end{definition}

In what follows we refer to condition (b) as the {\em minimality condition}.

We shall  also consider Reflected BSDEs with one lower (resp. upper) barrier.

\begin{definition}
We say that a triple $(Y,Z,K)$ of $\mathbb{F}$-adapted processes is a solution to  reflected backward stochastic differential equation
on the interval $[[\nu,\zeta]]$ with generator $f$, terminal value $\hat{\xi}$ and lower barrier $L$ (\underline{R}BSDE$^{\nu,\zeta}(\hat{\xi},f,L)$ for short) if
\begin{enumerate}
\item[(a)] $Y$ is c\`adl\`ag and $Z\in\mathcal{H}_{\mathbb{F}}(\nu,\zeta)$,
\item[(b)] $K\in\mathcal{V}^+_{0,\mathbb{F}}(\nu,\zeta)$, $L_t\le Y_t$, $t\in[\nu,\zeta]$, and
\begin{equation*}
\begin{split}
\int^{\zeta}_{\nu}(Y_{r-}-L_{r-})\,dK_r=0,
\end{split}
\end{equation*}
\item[(c)] $f(\cdot,Y,Z)$ is $\mathbb F$-progressively measurable and $\int^{\zeta}_{\nu}|f(r,Y_r,Z_r)|\,dr<\infty$,
\item[(d)] $Y_t=\hat{\xi}+\int^{\zeta}_t f(r,Y_r,Z_r)\,dr+(K_\zeta-K_t)-\int^{\zeta}_t Z_r\,dB_r$, $t\in[\nu,\zeta]$.
\end{enumerate}
\end{definition}

\begin{definition}
We say that a triple $(Y,Z,K)$ of $\mathbb{F}$-adapted processes is a solution to  reflected backward stochastic differential equation
on the interval $[[\nu,\zeta]]$ with generator $f$, terminal value $\hat{\xi}$ and upper barrier $U$ ($\mathrm{\overline{R}}$BSDE$^{\nu,\zeta}(\hat{\xi},f,U)$ for short) if $(-Y,-Z,K)$ is a solution to \underline{R}BSDE$^{\nu,\zeta}(\hat{\xi},-\tilde{f},-U$), where
\[
\tilde{f}(t,y,z)=f(t,-y,-z).
\]
\end{definition}

\begin{remark}
\label{rem.chv}
Let $a\in \mathbb R$. Observe that if  $(Y,Z,R)$ is an $\mathscr S$-solution to RBSDE$^T(\xi,f+dV,L,U)$, then $(\bar Y,\bar Z,\bar R)$ 
is a solution to RBSDE$^T(\bar\xi,\bar f+d\bar V,\bar L,\bar U)$,
where
\[
(\bar Y_t,\bar Z_t,\bar R_t):= (e^{a t} Y_t, e^{a t} Z_t,\int_0^te^{ar}\,dR_r),\quad \bar f(t,y,z):= e^{a t}f(t,e^{-a t}y,e^{-a t} z)-a y,
\]
\[
\bar \xi:= e^{a T}\xi,\quad \bar V_t:=\int_0^te^{ar}\,dV_r,\quad \bar L_t:=e^{at} L_t,\quad \bar U_t:=e^{at}U_t.
\]
Clearly, if $(\xi,f,V)$ satisfies any of conditions (H1), (H3)--(H5), then $(\bar\xi,\bar f,\bar V)$ satisfies it too.
If $f$ satisfies (H2), then $\bar f$ satisfies (H2) but with $\mu$ replaced by $\mu-a$, and if $f$ satisfies (Z), then $\bar f$
satisfies (Z) with $(\gamma,g_t)$ replaced by $(\gamma e^{a^+ T}, g_te^{\frac{-a^-t}{\kappa}})$.
Clearly $\bar L,\bar U$ are of class (D).
\end{remark}

Let us adopt the shorthand RBSDE$^{\zeta}:=$RBSDE$^{0,\zeta}$ (the same for one reflecting barrier case).

The following result follows from  \cite{K:BSM}, Theorem 7.1.

\begin{theorem}\label{10wrzesnia2}
Assume that \textnormal{(H1)-(H5), (Z)} are in force. Then there exists a  solution $(Y,Z,K)$ to \textnormal{\underline{R}BSDE}$^T(\xi,f,L)$ such that $Y$ is of class \textnormal{(D)}.
\end{theorem}

Now, we consider the so-called {\em weak Mokobodzki's condition}.

\begin{enumerate}
\item[(WM)] There exists a  semimartingale $X$ such that $L\le X\le U$.
\end{enumerate}

\begin{remark}\label{12kwietnia1}
Note that condition  (WM)  is easy  verified  in the case of c\`adl\`ag barriers $L,U$ satisfying $L_t<U_t$ and $L_{t-}<U_{t-}$, $t\in[0,T]$ (see e.g. \cite{Topolewski}, Lemma 3.1).
\end{remark}

The following results follow from  \cite{KRz3}, Theorem 7.3, Theorem 9.3.

\begin{theorem}\label{12wrzesnia1}
Assume that \textnormal{(H1)-(H5), (WM)} are in force.
Then there exists an $\mathscr S$-solution $(Y,Z,R)$ to \textnormal{RBSDE}$^T(\xi,f,L,U)$ 
such that $Y$ is of class \textnormal{(D)}. 
\end{theorem}

\begin{theorem}\label{dyn20}
Let  $(Y,Z,R)$ be an $\mathscr S$-solution to   \textnormal{RBSDE}$^T(\xi,f,L,U)$. Then for any $\theta\in\mathcal{T}$
\begin{equation}\label{11paz1}
Y_{\theta}=\mathop{\mathrm{ess\,inf}}_{\sigma\ge \theta}\mathop{\mathrm{ess\,sup}}_{\tau\ge \theta}\mathbb{E}^f_{\theta,\tau\wedge\sigma}(L_{\tau}\mathbf{1}_{\{\tau \le\sigma,\tau<T\}}+U_{\sigma}\mathbf{1}_{\{\sigma<\tau\}}+\xi\mathbf{1}_{\{\tau=\sigma=T\}}).
\end{equation}
\end{theorem}

Now we shall present two comparison theorems for solutions to Reflected BSDEs. 
The first one is a direct corollary to representation \eqref{11paz1}
and Proposition \ref{nonlinprop}. Thus, it requires all the assumptions (H1)--(H5), (Z).
The second one, with the proof based on the Tanaka--Meyer formula, requires merely assumptions (H1), (H2), (Z)
but we also need an additional  condition:
\begin{enumerate}
\item[(G)] $\mathbb E\int_0^Tf^-(s,L^{+,*}_s,0)\,ds+\mathbb E\int_0^Tf^+(s,-U^{-,*}_s,0)\,ds<\infty$.
\end{enumerate}

In the following proposition  $L^1,U^1,L^2,U^2$ are $\mathbb F$-adapted c\`adl\`ag processes 
of class (D) and $L^1\le U^1$, $L^2\le U^2$. Furthermore $\xi_1, \xi_2$
are $\mathcal F_T$-measurable random variables, with $L^1_T\le \xi^1\le U^1_T$,
$L^2_T\le \xi^2\le U^2_T$, $f^1,f^2$ are $Prog(\mathbb F)\otimes \mathcal B(\mathbb R)\otimes\mathcal B(\mathbb R^d)$
measurable real functions on $[0,T]\times \Omega\times\mathbb R\times\mathbb R^d$, and
$V^1,V^2\in\mathcal V_{0,\mathbb F}(0,T)$.

\begin{proposition}[{\em Comparison theorem for \textnormal{RBSDEs}}]\label{10wrzesnia1}
Let triple $(Y^i,Z^i,R^i)$ be an  $\mathscr S$-solution to \textnormal{RBSDE}($\xi^i,f_i+dV^i,L^i,U^i)$,
$i=1,2$.   Assume also that  $\xi^1\le\xi^2$, $dV^1\le dV^2$, $L^1\le
L^2$, $U^1\le U^2$.
\begin{enumerate}
\item[(i)] If $f_1,f_2,V^1,V^2$ satisfy \textnormal{(H1)--(H5), (Z)}, then  $Y^1_t\le Y^2_t$, $t\in[0,T]$.
\item[(ii)] If  $(Y^1-Y^2)^+\in \mathcal S^p_{\mathbb{F}}(0,T)$ for some $p>1$,  $f_2$ satisfies \textnormal{(H1), (H2)} and $f^{1}(\cdot,Y^1,Z^1)\le f^2(\cdot,Y^1,Z^2),\, dt\otimes \mathbb{P}$-a.e.,
then $Y^1_t\le Y^2_t$, $t\in[0,T]$.
\item[(iii)] If \textnormal{(G)} holds, $f_2$ satisfies \textnormal{(H1), (H2), (Z)} and $f^{1}(\cdot,Y^1,Z^1)\le f^2(\cdot,Y^1,Z^1),\, dt\otimes \mathbb{P}$-a.e.,
then $Y^1_t\le Y^2_t$, $t\in[0,T]$.
\end{enumerate}
\end{proposition}
\begin{proof}
(i) Follows from Theorem \ref{dyn20} and Proposition \ref{nonlinprop}(ii).

(ii)
Set $(Y,Z,R):=(Y^1-Y^2,Z^1-Z^2,R^1-R^2)$.
Let $\tau\in\mathcal T$. Set 
\[
\tau_\varepsilon:=\inf\{t\ge \tau: Y^1_t\le L^1_t+\varepsilon \text{ or } Y^2_t\ge U^2_t-\varepsilon\}.
\]
By the minimality condition $(Y^1,Z^1,R^{1,-})$ is a solution to 
$\overline{\text{R}}$BSDE$^{\tau,\tau_\varepsilon}(Y^1_{\tau_\varepsilon},f_1+dV^1,U)$
and $(Y^2,Z^2,R^{2,+})$ is a solution to 
$\underline{\text{R}}$BSDE$^{\tau,\tau_\varepsilon}(Y^2_{\tau_\varepsilon},f_2+dV^1,L)$.
Observe that $(Y,Z):=(Y^1-Y^2,Z^1-Z^2)$ is a solution to BSDE$^{\tau,\tau_\varepsilon}(Y_{\tau_\varepsilon},F+\rho+d(V^1-V^2)-d(R^{1,-}+R^{2,+}))$, where
\[
F(t,y,z):= f_2(t,y+Y^2_t,z+Z^2_t)-f_2(t,Y^2_t,Z^2_t),\quad \rho_t:=(f_1(t,Y^1_t,Z^1_t)-f_2(t,Y^1_t,Z^1_t)).
\]
Furthermore, $Y_{\tau_\varepsilon}\le\varepsilon$.
Let $(Y^\varepsilon,Z^\varepsilon)\in \mathcal S^2_{\mathbb{F}}(\tau,\tau_\varepsilon)\times \mathcal H^2_{\mathbb{F}}(\tau,\tau_\varepsilon)$
be a solution to BSDE$^{\tau,\tau_\varepsilon}(\varepsilon,F)$. 
By  \cite{K:BSM}, Proposition 3.1,  $0\le Y^\varepsilon\le \varepsilon$ on $[[\tau,\tau_\varepsilon]]$.
Thus, by  \cite{K:BSM}, Proposition 3.1, $Y\le Y^\varepsilon\le\varepsilon$ on $[[\tau,\tau_\varepsilon]]$. 
Consequently, $Y^+_\tau\le\varepsilon$ for any $\varepsilon>0$
and any $\tau\in \TT$. This  yields (ii).

(iii) The proof of (iii) is similar to (ii). Let $\tau\in\mathcal T$ and  $\tau_\varepsilon$
be as in the proof of (ii). By condition (G) and \cite{K:BSM}, Lemma 4.2, $Z^1,Z^2\in\mathcal H^{q}_{\mathbb{F}}(\tau,\tau_\varepsilon)$
for any $q\in (0,1)$.
As in (ii) the pair  $(Y,Z):=(Y^1-Y^2,Z^1-Z^2)$ is a solution to BSDE$^{\tau,\tau_\varepsilon}(Y_{\tau_\varepsilon},F+\rho+d(V^1-V^2)-d(R^{1,-}+R^{2,+}))$.
By the Tanaka--Meyer formula and assumptions that we made on the data, we have 
\[
\begin{split}
Y^+_\theta&\le \mathbb E^{\FF_\theta}\Big[Y^+_{\tau_\varepsilon}+\int_\theta^{\tau_\varepsilon}
|f_2(t,Y^2_t,Z^2_t)-f_2(t,Y^2_t,Z^1_t)|\,dt\Big]
\\&\le
\mathbb E^{\FF_\theta}\Big[\varepsilon+\gamma\int_\theta^{\tau_\varepsilon}
\Big((g_t+|Y^2_t|+|Z^2_t|)^\kappa+(g_t+|Y^2_t|+|Z^1_t|)^\kappa\Big)\,dt\Big].
\end{split}
\]
Consequently, by the Doob inequality, $Y^+\in \mathcal S^p_{\mathbb{F}}(\tau,\tau_\varepsilon)$ for some $p>1$.
From this as in (ii)
we infer  that   $Y^+_\tau\le\varepsilon$ for any $\varepsilon>0$
and any $\tau\in \TT$, which    yields (iii).

\end{proof}

\section{Asymptotic Mokobodzki's condition}
\label{roz5}

As an exception, in the present section  $T$ shall denote   a bounded stopping time.
Let us  consider the following {\em asymptotic Mokobodzki's condition}:
\begin{enumerate}
\item[(AM)] there exists an increasing sequence $(\tau_k)\subset \mathcal T$ such that $\tau_k\nearrow T$
and for any $k\ge 1$ there exists a c\`adl\`ag semimartingale $X^k$ such that $L_t\le X^k_t\le U_t,\, t\in [0,\tau_k]$.
\end{enumerate}

We adopt the convention that $Y_{0-}=0$ for any c\`adl\`ag process $Y$, and $[a,a)=\{a\}$
for any $a\in\mathbb R$.

\begin{definition}
\label{def.as}
Assume (AM). We say that   triple $(Y,Z,R)$ of 
$\mathbb{F}$-adapted processes is an $\mathscr S_a$-solution to  reflected backward stochastic differential equation
on the interval $[[0,T]]$ with generator $f$, terminal value $\xi$, 
lower barrier $L$ and upper barrier $U$ (RBSDE$^{T}(\xi,f,L,U)$ for short) if 
\begin{enumerate}
\item[(a)] $Y$ is c\`adl\`ag and $L_t\le Y_t\le U_t,\, t\in [0,T]$,
\item[(b)] for  any sequence $(\tau_k)$ of (AM) the triple $(Y,Z,R)$
is an $\mathscr S$-solution to the problem RBSDE$^{\tau_k}(Y_{\tau_k},f,L,U)$,
\item[(c)] $Y_T=\xi$, $Y_{T-}=(U_{T-}\wedge \xi)\vee L_{T-}$.
\end{enumerate}
\end{definition}

\begin{theorem}
\label{th.exas}
Assume that $f$ is independent of $z$-variable.
Furthermore, assume that \textnormal{(H1)--(H5)} hold and $L, U$ satisfy \textnormal{(AM)}. 
Then there exists a unique $\mathscr S_a$-solution
$(Y,Z,R)$ to \textnormal{RBSDE}$^T(\xi,f,L,U)$. 
Moreover, 
\begin{equation}
\label{eq4.rr}
\omega\in \{L_{T-}<U_{T-}\}\,\Rightarrow\, \exists_{t_\omega\in [0,T(\omega))}\,\, R^+_t=R^+_{t_\omega},\, t\in [t_\omega,T(\omega))\,\text{ or }\, R^-_t=R^-_{t_\omega},\, t\in [t_\omega,T(\omega)),
\end{equation}
and
\begin{equation}
\label{eq4.rr0}
\omega\in \{L_{T-}<\xi<U_{T-}\}\,\Rightarrow\, \exists_{t_\omega\in [0,T(\omega))}\,\, R_t=R_{t_\omega},\, t\in [t_\omega,T(\omega)).
\end{equation}
\end{theorem}
\begin{proof}
By \cite{K:SPA}, Theorem 3.11, there exists a c\`adl\`ag adapted process $Y$ of class (D)
that solves RBSDE$^T(\xi,f,L,U)$
in the sense of Definition \ref{df.main}. Let $(\tau_k)$ be a sequence of Definition \ref{def.as}.
By \cite{K:SPA}, Theorem 4.5, for any $k\ge 1$, $Y$ is a semimartingale on $[[0,\tau_k]]$ and $(Y,Z^k,R^k)$
is an $\mathscr S$-solution to RBSDE$^{\tau_k}(Y_{\tau_k},f,L,U)$, where
\[
Y_t=Y_{\tau_k}+\int_t^{\tau_k}f(s,Y_s)\,ds+(R^k_{\tau_k}-R^k_t)-\int_t^{\tau_k}Z^k_s\,ds,\quad t\in [0,\tau_k].
\]
By a standard uniqueness argument we deduce that
\[
R^k=R^{k+1},\, Z^k=Z^{k+1}\quad\text{on}\quad [[0,\tau_k]].
\]
Therefore, we may define processes $R,Z$ as follows
\[
R_t:=R^k_t,\, t\in [0,\tau_k],\quad Z_t:=Z^k_t,\, t\in [0,\tau_k].
\]
Clearly, the triple $(Y,Z,R)$ satisfies (a), (b) of Definition \ref{def.as} and $Y_T=\xi$.
What is left is to show that the second condition of Definition \ref{def.as}(c) holds.
By \cite{K:SPA}, Theorem 3.9, $\hat Y\le Y\le\bar Y$ on $[[0,T]]$, where $(\hat Y,\hat Z,\hat K)$
is a solution to $\overline{\text{R}}$BSDE$^T(\xi,f,U)$ and 
$(\bar Y,\bar Z,\bar K)$
is a solution to $\underline{\text{R}}$BSDE$^T(\xi,f,L)$. 
By the definitions of  solutions of the said two problems 
\[
\Delta \bar Y_T=-(\xi-L_{T-})^-,\quad \Delta \hat  Y_T=(\xi-U_{T-})^+,
\]
\[
(\bar Y_{T-}-L_{T-})\Delta \bar K_T=0,\quad
(U_{T-}-\hat Y_{T-})\Delta \hat K_T=0.
\] 
From this and the fact that $\hat Y\le Y\le\bar Y$ on $[[0,T]]$ one easily concludes $Y_{T-}=(U_{T-}\wedge \xi)\vee L_{T-}$.
Thus, $(Y,Z,R)$ is an $\mathscr S_a$-solution to RBSDE$^T(\xi,f,L,U)$.

The uniqueness part. Suppose that $(Y',Z',R')$ is an $\mathscr S_a$-solution to RBSDE$^T(\xi,f,L,U)$. 
Let $(\tau_k)\subset \TT$ be a non-decreasing sequence that satisfies (AM). 
By the Tanaka--Meyer formula and (H2)
\[
\mathbb E|Y_\sigma-Y'_\sigma|\le \mathbb E|Y_{\tau_k}-Y'_{\tau_k}|,\quad k\ge 1,\, \sigma\in\TT^{\tau_k}.
\]
Letting $k\to \infty$ and using Definition \ref{def.as}(c) yields $Y=Y'$.


Suppose, towards a   contradiction, that property \eqref{eq4.rr} does not hold. Then, for fixed $\omega\in \{L_{T-}<U_{T-}\}$, there exist non-decreasing sequences $(s_n), (t_n)\subset [0,T)$ such that $t_n, s_n\to T$ and  
\[
R^+_{t_{n+1}}>R^+_{t_n},\quad R^{-}_{s_{n+1}}>R^{-}_{s_{n}},\quad n\ge 1.
\]
By the minimality condition, $\limsup_{s\to T-}Y_s=U_{T-}$ and $\liminf_{s\to T-}Y_s=L_{T-}$.
Thus, $U_{T-}=L_{T-}$, a contradiction.
The proof of \eqref{eq4.rr0} is similar. Negating  \eqref{eq4.rr0} and applying the minimality condition
yields that  $\limsup_{s\to T-}Y_s=U_{T-}$ or  $\liminf_{s\to T-}Y_s=L_{T-}$.
This means that  $U_{T-}=\xi$ or $L_{T-}=\xi$, a contradiction.
\end{proof}

\begin{corollary}
\label{cor.ex1}
If $L,U$ satisfy \textnormal{(AM)},  then there exists a c\`adl\`ag process $X$
such that $L\le X\le U$ on $[[0,T]]$, and for any $\sigma\in \TT$ such that
 $[[0,\sigma]]\subset [[0,T[[\cup[[T_{\{L_{T-}<U_{T-}\}}]]$, $X$ is a semimartingale on $[[0,\sigma]]$.
\end{corollary}
\begin{proof}
Set $\xi:=\frac{1}{2}(L_{T-}+U_{T-})$. Let $(Y,Z,R)$ be an $\mathscr S_a$-solution to RBSDE$^T(\xi,0,L,U)$, 
and $\sigma$ be a stopping time from the assertion. By Theorem \ref{th.exas}
$R^{\sigma}\in\VV_{0,\mathbb F}(0,T)$, where $R^{\sigma}_t:=R_{t\wedge \sigma},\, t\in [0,T)$
and $R^\sigma_T:= \lim_{s\to T-} R^\sigma_s$. Let $(\sigma_k)$ be an announcing sequence for $\sigma$
and $(\tau_l)$ be a chain on $[[0,\sigma]]$ such that $\mathbb E|R^\sigma|_{\tau_l}<\infty,\, l\ge 1$.
By the definition of an $\mathscr S_a$-solution to RBSDE$^T(\xi,0,L,U)$, 
\[
\int_0^t Z_s\,dB_s=Y_t-Y_0+R_t,\quad t\in [0,T).
\]
By \cite{bdh}, Lemma 6.1, for any $q\in (0,1)$ there exists $c_q$ such that 
\[
\mathbb E\Big(\int_0^{\sigma_k\wedge\tau_l}|Z_s|^2\,ds\Big)^q
\le c_q\Big( \mathbb E|Y_{\sigma_k\wedge\tau_l}-Y_0+R_{\sigma_k\wedge\tau_l}| \Big)^q
\le c_q\Big(2\|Y\|_{\mathcal D^1}+ \mathbb E|R^\sigma|_{\tau_l} \Big)^q.
\]
From this and the fact that $(\tau_l)$ is a chain, one easily concludes that $\mathbf 1_{[0,\sigma]}Z\in\HH_{\mathbb F}(0,T)$.
Therefore, $Y$ is a semimartingale on $[[0,\sigma]]$. 

\end{proof}

\section{Mokobodzki's stochastic intervals}
\label{roz6}

Let us introduce the following notion. For any  c\`adl\`ag $\mathbb F$-adapted processes $X^1,X^2$
and   $\tau,\zeta\in\mathcal A$ we let 
\[
\Sigma_{\tau,\zeta}(X^1,X^2):=\{\sigma\in \mathcal A_\tau: \exists \text{ c\`adl\`ag semimartingale } X \text{ such that } 
X^1_t\le X_t\le X^2_t,\, t\in [\tau,\zeta]\}.
\]
We also use the  shorthand $\Sigma_{\tau}(X^1,X^2):= \Sigma_{\tau,T}(X^1,X^2)$.

We shall start with the following result.

\begin{proposition}
\label{prop.m1}
Let  $\tau,\zeta\in\mathcal{A}$,   $\tau\le\zeta$, and $L,U$ be of class \textnormal{(D)} on $[[\tau,\zeta]]$.
Then, there exists an increasing sequence $(\sigma_k)\subset \mathcal A_\tau$
such that $\sigma_k\nearrow \sigma^*:=\esssup \Sigma_{\tau,\zeta}(L,U)$. As a result, $\sigma^*$
is a stopping time.
\end{proposition}
\begin{proof}
By \cite{KS}, Theorem A.3, there exists a random variable $\sigma^*$ such that 
$\sigma^*=\esssup\Sigma_{\tau,\zeta}(L,U)$, i.e. $\sigma\le \sigma^*,\, \sigma\in\Sigma_{\tau,\zeta}(L,U)$,
and for  any random variable $\hat \sigma$ with the property: $\sigma\le \hat \sigma,\, \sigma\in\Sigma_{\tau,\zeta}(L,U)$,
we have $\sigma^*\le\hat \sigma$.  We shall prove that $\sigma^*$ is a stopping time.
To do this end,  we show that the family $\Sigma_{\tau,\zeta}(L,U)$  is closed under pairwise maximization i.e. 
$\sigma_1,\sigma_2\in\Sigma_{\tau,\zeta}(L,U)$ implies $\sigma_1\vee\sigma_2\in\Sigma_{\tau,\zeta}(L,U)$. 

First, observe that without loss of generality $L$ may be assumed to be non-negative. 
Indeed,  for any semimartingale $X$, we have
\[
\Sigma_{\tau,\zeta}(L,U)=\Sigma_{\tau,\zeta}(L+X,U+X).
\]
Thus, $\esssup \Sigma_{\tau,\zeta}(L,U)=\esssup \Sigma_{\tau,\zeta}(L+X,U+X)$.
On the other hand, since $L$ is of class (D) on $[[\tau,\zeta]]$, there exists a martingale $M$ of class (D) on $[[\tau,\zeta]]$
such that $L+M\ge 0$ (see \cite{DM}, Theorem 24, page 419).

Assume that $L\ge 0$.
Let $\sigma_1,\sigma_2\in \Sigma_{\tau,\zeta}(L,U)$. 
Then there exist processes c\`adl\`ag semimartingales $X^1,X^2$ such that  
$L_t\le X^1_t\le U_t$, $t\in[\tau,\sigma_1]$ and  
$L_t\le X^2_t\le U_t$, $t\in[\tau,\sigma_2]$. 
Clearly $\tilde{X}^i_t:=X^i_{t\vee \tau}\mathbf{1}_{[0,\sigma_i)}(t)$, $i=1,2$ are c\`adl\`ag semimartingales as well,
which in turn implies that   $\tilde{X}:=\tilde{X}^1\vee\tilde{X}^2$ is a c\`adl\`ag semimartingale.
Let us define $A_t:=\mathbf1_{\{t\ge \sigma_1\vee\sigma_2\}}L_{\sigma_1\vee\sigma_2}$.
Clearly $A$ is an increasing adapted c\`adl\`ag process.
Thus, $\hat X:= \tilde X+A$ is a c\`adl\`ag semimartingale.
Observe that  $L_t\le\hat{X}_t\le U_t$, $t\in[\tau,\sigma_1\vee\sigma_2]$. 
Therefore, $\sigma_1\vee\sigma_2\in\Sigma_{\tau,\zeta}(L,U)$. 
By \cite{KS}, Theorem A.3,  there exists a nondecreasing sequence $(\sigma_n)\subset \Sigma_{\tau,\zeta}(L,U)$ such that $\sigma^*=\lim_{n\to\infty}\sigma_n$ a.s. Thus,     $\sigma^*$ is a stopping time.
\end{proof}

Based on the above result for any finite $\zeta\in \mathcal A$ we may define the mapping
\[
\mathcal A^\zeta\ni\tau\longmapsto\mathring\tau^\zeta_{L,U}:= \esssup\Sigma_{\tau,\zeta}(L,U)\in \mathcal A_\tau^\zeta.
\] 
We shall  use in the paper the shorthand notations $\mathring\tau^\zeta: =\mathring\tau^\zeta_{L,U}$
and  $\mathring\tau :=\mathring\tau^T$. 
We use the full notation only in case we want to highlight the dependence of theses objects   on the barriers and terminal time
or when we use a triple different    from  $(L,U,T)$.
We also denote $\mathring\tau^\zeta_X:=\mathring\tau^\zeta_{X,X}$
for any $\tau\in\AAA^\zeta$ and c\`adl\`ag $\mathbb F$-adapted process $X$ of class (D) on $[[0,\zeta]]$.

\begin{proposition}
\label{prop.ex2}
For any finite $\tau\in\TT$ there exists a c\`adl\`ag process 
$X$ such that $L\le X\le U$ on $[[\tau,\mathring\tau]]$,  and for any $\sigma\in \TT_\tau^{\mathring\tau}$
with  $[[\tau,\sigma]]\subset [[\tau,\mathring\tau[[\cup [[\mathring\tau_{\{L_{\mathring\tau-}<U_{\mathring\tau-}\}}]]$, 
$X$ is a semimartingale on $[[\tau,\sigma]]$.
\end{proposition}
\begin{proof}
It is enough to apply  Corollary \ref{cor.ex1} with $T=\mathring\tau$. 
\end{proof}

For $\tau\in\TT$ we let 
\[
\mathscr C_\tau:=\{\mathring\tau^{T+1}_{L^{\mathring\tau},U^{\mathring\tau}}=T+1\}.
\]
We call $\mathscr M_{L,U}(\tau):= [[\tau,\mathring\tau[[\cup[[\mathring\tau_{\mathscr C_\tau}]]$  Mokobodzki's stochastic intervals.
In case $L=U$ we write $\mathscr M_{L}(\tau):=\mathscr M_{L,U}(\tau)$.

\begin{theorem}
\label{prop7.3m}
For any $\tau\in\mathcal T$, we have:
\begin{enumerate}
\item[(i)] if $\sigma\in \Sigma_{\tau}(L,U)$, then $[[\tau,\sigma]]\subset \mathscr M_{L,U}(\tau)$,
\item[(ii)]
\[
L_{\mathring\tau-}=U_{\mathring\tau-}\,\text{ or }\,L_{\mathring\tau}=U_{\mathring\tau} \quad\text{on}\quad \{\tau<\mathring\tau<T\}.
\]
\item[(iii)]
\[
 \bigcap_{\sigma\in\Sigma_\tau(L,U)} \{\sigma<\mathring\tau\}\subset \{L_{\mathring\tau-}=U_{\mathring\tau-}\},
 \quad \mathscr C^c_\tau\subset  \{L_{\mathring\tau-}=U_{\mathring\tau-}\},
\]
\item[(iv)] there exists a non-decreasing sequence of stopping times $(\tau_k)\subset \TT_\tau$
such that
\[
\bigcup_{k\ge 1}[[\tau,\tau_k]]=\mathscr M_{L,U}(\tau),
\]
\item[(v)] there exists a c\`adl\`ag process $X$ on $[0,T]$ that is a semimartingale on $[[\tau,\sigma]]$
for any stopping time $\sigma\in\TT_\tau$ satisfying $[[\tau,\sigma]]\subset \mathscr M_{L,U}(\tau)$,
\item[(vi)] for any $\tau\in \TT$ we have
\[
\mathscr C_\tau\cap\{\mathring\tau<T\}\subset \{L_{\mathring\tau}=U_{\mathring\tau}\},
\]
\item[(vii)] $\gamma_\tau\le \mathring\tau$ and for any $k\ge 1$, we have $[[\tau,\gamma_\tau^k]]\subset\mathscr M_{L,U}(\tau)$
(cf \eqref{eq2.3cdf}).
\end{enumerate}
\end{theorem}
\begin{proof}
Clearly, without loss of generality, we may assume that $L,U$ are bounded. 

Ad (i). Suppose that $\sigma\in \Sigma_{\tau}(L,U)$. 
Let  $A:=\{\sigma<\mathring\tau\}\in\FF_{\mathring\tau}\cap \FF_\sigma$
and $X$ be a c\`adl\`ag semimartingale such that $L\le X\le U$ on $[[\tau,\sigma]]$. 
Let us define
\[
Y_t:=X_t\mathbf1_{[\tau,\sigma)}(t)+\mathbf1_{A^c}\mathbf1_{[\mathring\tau,T+1]}(t)X_{\mathring\tau},\quad t\ge 0.
\]
Then, clearly, $Y$ is a semimartingale and $L^{\mathring\tau}_t\le Y_t\le U^{\mathring\tau}_t,\, t\in [\tau, \sigma_{A}\wedge(T+1)]$.
Thus, 
\[
\sigma_{A}\wedge(T+1)\le \esssup \Sigma_{\tau,T+1}(L^{\mathring\tau},U^{\mathring\tau}).
\]
Consequently, $A^c\subset \mathscr C_\tau$. Therefore
\[
\begin{split}
[[\tau,\sigma]]&=([[\tau,\sigma]]\cap A)\cup ([[\tau,\sigma]]\cap A^c)
\\&\subset ([[\tau,\mathring\tau[[\cap A)\cup ([[\tau,\mathring\tau[[\cap A^c)\cup( [[\mathring\tau]]\cap A^c)
\\&\subset [[\tau,\mathring\tau[[\cup ( [[\mathring\tau]]\cap \mathscr C_\tau)=\mathscr M_{L,U}(\tau).
\end{split}
\]

Ad (ii). We set $L'_t:= L_{t\wedge T}\mathbf1_{\{t\ge \tau\}}$ and $U'_t:= U_{t\wedge T}\mathbf1_{\{t\ge \tau\}}$. 
Let 
\[
A:=\{L_{\mathring\tau-}<U_{\mathring\tau-},\, L_{\mathring\tau}<U_{\mathring\tau}, \, \tau<\mathring\tau<T\}
=\{L'_{\mathring\tau-}<U'_{\mathring\tau-},\, L'_{\mathring\tau}<U'_{\mathring\tau}, \, \tau<\mathring\tau<T\}.
\]
Striving for a contradiction,
suppose that $\mathbb{P}(A)>0$. Set $\xi:=\frac12(L'_{\mathring\tau-}+U'_{\mathring\tau-})$. By Theorem  \ref{th.exas},
there exists a unique $\mathscr S_a$-solution $(Y,Z,R)$ to RBSDE$^{\mathring\tau}(\xi,0,L',U')$. 
Moreover, for any $\omega\in A$
there exists $t_\omega\in [0,\mathring\tau(\omega))$ such that $R_t(\omega)=R_{t_\omega}(\omega),\, t\in [t_\omega,\mathring\tau(\omega))$. From this we deduce that for any announcing sequence $(\tau_k)$ for $\mathring\tau_{|A^c}$
we have that $R$ is a finite variation c\`adl\`ag process on $[[0,\tau_k\wedge \mathring\tau_{|A}]]$. Let $(\delta_n^k)$
be a chain on $[[0,\tau_k\wedge \mathring\tau_{|A}]]$ such that $\mathbb E|R|^2_{\delta_n^k},\, n\ge 1$.
Then
\[
\mathbb E\int_0^{\delta^k_n}|Z_s|^2\,ds<\infty,\,\quad n,k\ge 1.
\]
Thus, $Y$ is a semimartingale on $[[0,\delta_n^k]]$ for any $n,k\ge 1$.
Since  $(\tau_k)$ is an announcing sequence for $\mathring\tau_{|A^c}$, there exists $k_0$
such that $\mathbb{P}(\tau_{k_0}\wedge \mathring\tau_{|A}=\mathring\tau_{|A})>0$.  Thus, by the definition of a chain,
there exists $n_{k_0}$ such that $\mathbb{P}(\delta^{k_0}_{n_{k_0}}=\mathring\tau_{|A})>0$. 
Set $B:=\{\delta^{k_0}_{n_{k_0}}=\mathring\tau_{|A}\}$, and 
\[
\zeta^B:=\inf\Big\{t\ge (\delta^{k_0}_{n_{k_0}})_{|B}: \frac12(L'_{\mathring\tau}+U'_{\mathring\tau})\notin [L_t,U_t]\Big\}.
\]
Observe that $\zeta^B>\delta^{k_0}_{n_{k_0}}=\mathring\tau$ on the set $B$.
We let (notice that $B\in\mathcal F_{\delta^{k_0}_{n_{k_0}}}$)
\[
\bar Y_t:=\mathbf1_{\{t<\delta^{k_0}_{n_{k_0}}\}}Y_t+\mathbf1_B\mathbf1_{[\delta^{k_0}_{n_{k_0}},\zeta^B)}(t) \frac12(L'_{\mathring\tau}+U'_{\mathring\tau}).
\]
Clearly, $\bar Y$ is a c\`adl\`ag semimartingale and 
$L'_t\le \bar Y_t\le U'_t,\, t\in [0, (\delta^{k_0}_{n_{k_0}})_{|B^c}\wedge \zeta^B)$.
Since on the set $B$ we have $(\delta^{k_0}_{n_{k_0}})_{|B^c}\wedge \zeta^B=\zeta^B>\mathring\tau$,
we arrive at contradiction. 

Ad (iii). Let $A:=\{L'_{\mathring\tau-}<U'_{\mathring\tau-}\}$ ($L'$, $U'$ defined as in the proof of (ii)). With this choice of $A$, we may  repeat the part of the reasoning of the proof of (ii)
 that led to the statement that   $Y$ is a semimartingale on $[[0,\delta_n^k]]$ for any $n,k\ge 1$.
Obviously, $\delta^k_n\vee\tau\in\Sigma_{\tau}(L',U')=\Sigma_{\tau}(L,U),\, n,k\ge 1$. 
By the construction of $\delta^k_n$   we easily get that 
\[
\{L'_{\mathring\tau-}<U'_{\mathring\tau-}\}=A=\bigcup_{n,k\ge 1}\{\delta^k_n\vee\tau=\mathring\tau_{|A}\}\subset
\bigcup_{\sigma\in\Sigma_\tau(L',U')} \{\sigma=\mathring\tau\}=\bigcup_{\sigma\in\Sigma_\tau(L,U)} \{\sigma=\mathring\tau\}.
\]
Hence,
\[
\bigcap_{\sigma\in\Sigma_\tau(L,U)} \{\sigma<\mathring\tau\}\subset \{L'_{\mathring\tau-}=U'_{\mathring\tau-}\}.
\]
Consequently, intersecting both sides of the above inclusion with $\{\tau<\mathring\tau\}$ yields
\[
\begin{split}
\bigcap_{\sigma\in\Sigma_\tau(L,U)} \{\sigma<\mathring\tau\}&=
\{\tau<\mathring\tau\}\cap \bigcap_{\sigma\in\Sigma_\tau(L,U)} \{\sigma<\mathring\tau\}\\&\subset 
\{\tau<\mathring\tau\}\cap \{L'_{\mathring\tau-}=U'_{\mathring\tau-}\}\subset \{L_{\mathring\tau-}=U_{\mathring\tau-}\}.
\end{split}
\]
This finishes the proof of the first inclusion in  (iii).  The second one is a consequence of Proposition \ref{prop.ex2}.
Assertion (iv) follows directly from Proposition \ref{prop.m1} applied to 
$\Sigma_{\tau,T+1}(L^T,U^T)$.

Ad (v). 
By \cite{K:SPA}, Theorem 3.11, there exists a solution $Y$ to RBSDE$^{T+1}(L'_{T+1},0,L',U')$
in the sense of Definition \ref{df.main}. Let $(\tau_k)\subset \Sigma_{\tau,T+1}(L^{\mathring\tau}, U^{\mathring\tau})$ be a 
non-decreasing sequence of stopping times such that
$\tau_k\nearrow \tau^*:=\esssup\Sigma_{\tau,T+1}(L^{\mathring\tau}, U^{\mathring\tau})=\esssup\Sigma_{0,T+1}(L',U')$.
By \cite{K:SPA}, Theorem 4.5,  $Y$ is a semimartingale on $[[0,\tau_k]]$ for any $k\ge 1$. Thus
\[
Y_t=Y_0-R^k_t+M^k_t,\quad t\in [0,\tau_k],
\]
where $M^k$ is a martingale and $R^k$ is a finite variation process from the Doob--Meyer
decomposition of $Y$.  By the representation property of the filtration $\mathbb F$,
\[
M^k_t=\int_0^tZ^k_s\,dB_s,\quad t\in [0,\tau_k],
\] 
for a progressively measurable process $Z^k$. By the uniqueness of the Doob--Meyer decomposition we may  define processes $R,Z$ on $[[0,\tau^*[[$
as follows
\[
R_t:= R^k_t,\quad Z_t:=Z^k_t,\quad t\in [0,\tau^*).
\]
Thus, in particular,
\[
Y_t=Y_0-R_t+\int_0^tZ_s\,dB_s,\quad t\in [0,\tau^*)\cap [0,\mathring\tau].
\] 
Observe that $\mathring\tau\le \tau^*$ and   $\tau^*=T+1$ on $\mathscr C_\tau$. Therefore
\[
Y=Y_0-R+\int_0^\cdot Z_s\,dB_s\quad\text{on}\quad \mathscr M_{L,U}(\tau).
\]
Ad (vi). Suppose, striving for a contradiction, that 
\[
\mathbb{P}(\mathscr C_\tau\cap\{\mathring\tau<T\}\cap \{L_{\mathring\tau}<U_{\mathring\tau}\})>0
\]
and denote the set under the probability by $A$. Clearly $A\in\FF_{\mathring\tau}$. 
Let $(\tau_k)$ be the sequence of (iv).  
Then, by the definition of $\mathscr M_{L,U}(\tau)$ there exists $k_0$ such that $\mathbb{P}(\{\tau_{k_0}=\mathring\tau\}\cap A)>0$.
Let $X$ be an $\mathbb F$-adapted  c\`adl\`ag process that is a semimartingale on $[[\tau,\tau_{k_0}]]$
and lies between the barriers $L,U$. Set 
\[
\hat X:= X\mathbf1_{[0,\tau_{k_0})}+\frac12\mathbf1_{[\tau_{k_0},T]}\mathbf(L_{\tau_{k_0}}+U_{\tau_{k_0}}).
\]
Clearly, $\hat X$ is a c\`adl\`ag $\mathbb F$-adapted process that is a semimartingale on $[[\tau,T]]$.
Moreover, $\hat X$ lies between the barriers on $[[0,\sigma^*[[$, where
\[
\sigma^*:=\inf\{t\ge\tau_{k_0}: \frac12\mathbf(L_{\tau_{k_0}}+U_{\tau_{k_0}})\notin [L_t,U_t]\}\wedge T.
\]
Observe that $\sigma^*>\mathring\tau$ on $A\cap \{\tau_{k_0}=\mathring\tau\}$, which contradicts 
the definition of $\mathring\tau$.
\end{proof}

\begin{theorem}
Assume \textnormal{(AM)}. Let $(Y,Z,R)$ be a triple that satisfies \textnormal{(a), (b)} of Definition \textnormal{\ref{def.as}} and $Y_T=\xi$.
Then $Y$ satisfies \textnormal{(c)} of Definition \textnormal{\ref{def.as}}, in other words, $(Y,Z,R)$ is an $\mathscr S_a$-solution
to \textnormal{RBSDE}$^T(\xi,f,L,U)$.
\end{theorem}
\begin{proof}
By (AM) $\mathring\tau=T$ for any $\tau\in\TT$.
Let $(\tau_k)$ be the sequence of Proposition \ref{prop7.3m}(iv) for $\tau=0$.
By the assumptions that we made, $Y$ is an $\mathscr S$-solution
to RBSDE$^{\tau_k}(Y_{\tau_k},f,L,U)$ for any $k\ge 1$. Thus,
$Y_{\tau_k-}=(Y_{\tau_k}\wedge U_{\tau_k-})\vee L_{\tau_k-},\, k\ge 1$.
Since $\bigcup_{k\ge 1}\{\tau_k=T\}=F_0^c$, we get that $Y_{T-}=(\xi\wedge U_{T-})\vee L_{T-}$
on $F_0^c$. On the other hand, by Proposition \ref{prop7.3m}(iii) $F_0\subset \{L_{T-}=U_{T-}\}$.
This implies that on $F_0$, $Y_{T-}=(\xi\wedge U_{T-})\vee L_{T-}=L_{T-}=U_{T-}$.
\end{proof}

\section{Non-semimartingale solutions}
\label{roz7}

Thanks to the results of the previous section
for any $\tau\in \TT$ and c\`adl\`ag $\mathbb F$-adapted process $X$ of class (D) 
there exists a non-decreasing sequence  $(\tau_k)\subset \TT_\tau$ such that $\tau_k\nearrow \mathring\tau_X$
and $\bigcup_{k\ge 1}[[\tau,\tau_k]]= \mathscr M_{X}(\tau)$. By the Doob--Meyer decomposition for each $k\ge 1$
there exists a unique par $(V^{[\tau],k},Z^{[\tau],k})\in \mathcal V_{0,\mathbb{F}}(\tau,\tau_k)\times \HH_{\mathbb{F}}(\tau,\tau_k)$ such that
\[
X_t=X_\tau+V^{[\tau],k}_t+\int_\tau^tZ^{[\tau],k}_s\,dB_s,\quad t\in [\tau,\tau_k].
\]
It is clear that $(V^{[\tau],k}, Z^{[\tau],k})=(V^{[\tau],k+1}, Z^{[\tau],k+1})$ on $[[\tau,\tau_k]]$.
Thus, we may define a pair of processes $(V^{[\tau]},Z^{[\tau]})$ on $\mathscr M_X(\tau)$ by 
\[
(V^{[\tau]},Z^{[\tau]}):=(V^{[\tau],k}, Z^{[\tau],k})\quad\text{on}\quad [[\tau,\tau_k]],\, k\ge 1.
\]
We let
\begin{equation}
\label{eq9.5.semfff}
\mathscr Z^{[\tau]}(X):= Z^{[\tau]},\quad \mathscr V^{[\tau]}(X):= V^{[\tau]}.
\end{equation}
Therefore for any c\`adl\`ag $\mathbb F$-adapted process $X$ and $\tau\in\TT$, we have
\begin{equation}
\label{eq9.5.sem}
X_t=X_\tau +\mathscr V_t^{[\tau]}(X)+\int_\tau^t\mathscr Z_s^{[\tau]}(X)\,dB_s,\quad t\in \mathscr M_X(\tau).
\end{equation}
In what follows for given  processes $U,X\in \HH_{\mathbb F}(0,T)$ we denote 
\[
F_{X,U}(t):= \int_0^tf(s,X_s,U_s)\,ds,\quad t\in [0,T].
\]

\begin{definition}\label{12wrzesnia2}
We say that a c\`adl\`ag $\mathbb F$-adapted  process $Y$ is a
solution to reflected backward stochastic differential equation 
on  $[0,T]$  with the terminal value $\xi$, generator $f$, lower barrier $L$ and upper barrier $U$ 
(RBSDE$^T(\xi,f,L,U)$ for short) if for any  $\tau\in\mathcal{T}$
and $\sigma\in \Sigma_{\tau}(L,U)$,  $(Y,\mathscr Z^{[\tau]}(Y),\mathscr V^{[\tau]}(Y+F_{Y,\mathscr Z^{[\tau]}(Y)}))$ is an $\mathscr S$-solution to RBSDE$^{\tau,\sigma}(Y_{\sigma},f,L,U)$.
\end{definition}

\begin{remark}
Assume that (WM) is satisfied. Let $Y$ be a solution to RBSDE$^T(\xi,f,L,U)$. Then $Y$ is an $\mathscr S$-solution to RBSDE$^T(\xi,f,L,U)$, due to Definition \ref{10wrzesnia3}.
\end{remark}

Let $\xi_1,\xi_2, L^1,U^1,L^2,U^2,f_1,f_2,V^1,V^2$ be the set of data  described
in the comment preceding   
 Proposition \ref{10wrzesnia1}. 

\begin{theorem}[{\em Comparison theorem for \textnormal{RBSDEs}}]
\label{th.cop1}
Let  process $Y^i$ be a solution to \textnormal{RBSDE}($\xi^i,f_i+dV^i,L^i,U^i)$,
$i=1,2$.   Assume also that  $\xi^1\le\xi^2$, $dV^1\le dV^2$, $L^1\le
L^2$, $U^1\le U^2$.
\begin{enumerate}
\item[(i)] If  $(Y^1-Y^2)^+\in \mathcal S^p_{\mathbb{F}}(0,T)$ for some $p>1$,  $f_2$ satisfies \textnormal{(H1), (H2)} and \[
f^{1}(\cdot,Y^1,Z^1)\le f^2(\cdot,Y^1,Z^2),\quad dt\otimes \mathbb{P}\text{-a.e.},
\]
then $Y^1_t\le Y^2_t$, $t\in[0,T]$.
\item[(ii)] If \textnormal{(G)} holds (cf Proposition \ref{10wrzesnia1}), $f_2$ satisfies \textnormal{(H1), (H2), (Z)} and 
\[
f^{1}(\cdot,Y^1,Z^1)\le f^2(\cdot,Y^1,Z^1),\quad dt\otimes \mathbb{P}\text{-a.e.},
\]
then $Y^1_t\le Y^2_t$, $t\in[0,T]$.
\end{enumerate}
\end{theorem}
\begin{proof}
For $\tau\in\TT$, we let  
$\tau_\varepsilon:=\inf\{t\ge \tau: Y^1_t\le L^1_t+\varepsilon \text{ or } Y^2_t\ge U^2_t-\varepsilon\}\wedge T$. 
Observe that, by Proposition \ref{prop7.3m}(iii), for any $\tau\in\TT$
\[
[[\tau,\tau_\varepsilon]]\subset \mathscr M_{L,U}(\tau).
\]
Thus $(Y^i,\mathscr Z^{[\tau]}(Y^i),\mathscr V^{[\tau]}(Y^i+F^i_{Y^i,\mathscr Z^{[\tau]}(Y^i)}))$ is an $\mathscr S$-solution to RBSDE$^{\tau,\tau_\varepsilon}(Y^i_{\tau_\varepsilon},f,L,U)$
for $i=1,2$, $\tau\in\TT$ and $\varepsilon>0$. Therefore the proofs of Proposition \ref{10wrzesnia1}(ii)
and Proposition \ref{10wrzesnia1}(iii) apply to $Y^1,Y^2$.
\end{proof}

\begin{corollary}
The following hold.
\begin{enumerate}
\item[(i)] Assume {\rm (H1), (H2), (G), (Z)}. Then there exists at most one solution 
to the problem \textnormal{RBSDE}$^T(\xi,f,L,U)$.
\item[(ii)] Assume {\rm (H1), (H2)}. Then there exists at most one solution $(Y,Z,R)$
to the problem \textnormal{RBSDE}$^T(\xi,f,L,U)$ such that $Y\in \mathcal S^p_{\mathbb{F}}(0,T)$.
\end{enumerate}
\end{corollary}

Now, we shall proceed to the main existence result for  RBSDEs. 
Recall  that by Proposition  \ref{prop7.3m}, 
$\gamma_{\tau}\le\mathring\tau$ (cf \eqref{eq2.1}) and $[[\tau,\gamma_\tau^k]]\subset \mathscr M_{L,U}(\tau)$
for any $\tau\in\TT$ and $k\ge 1$ (cf \eqref{eq2.3cdf}). 
This will be the crucial point in  the proof of the following theorem.

\begin{theorem}\label{9wrzesnia1}
Assume that \textnormal{(H1)-(H5), (Z)} are in force. 
Then there exists a  unique solution $Y$ of \textnormal{RBSDE}$^T(\xi,f,L,U)$ such that $Y$ is of class \textnormal{(D)}. 
Moreover, the following hold.
\begin{enumerate}
\item[(i)]  Let $\{\xi_n\}$ be a sequence of integrable $\mathcal{F}_T$-measurable random variables such that $\xi_n\nearrow\xi$ and let
\[
\bar{f}_n(t,y,z)=f(t,y,z)+n(y-L_t)^-.
\]
Then for every $n\ge 1$ there exists a solution $(\bar{Y}^n,\bar{Z}^n,\bar{A}^n)$ to $\mathrm{\overline{R}}$\textnormal{BSDE}$^T(\xi_n,\bar{f}_n,,U)$ such that $\bar{Y}^n$ is of class \textnormal{(D)}. Moreover,   $\bar{Y}^n_t\nearrow Y_t$ 
and $d\bar{A}^n\le d\bar{A}^{n+1},\, n\ge 1$.
\item[(ii)] Let $\{\xi_n\}$ be a sequence of integrable $\mathcal{F}_T$-measurable random variables such that $\xi_n\searrow\xi$ and let
\[
\tilde{f}_n(t,y,z)=f(t,y,z)-n(y-U_t)^+.
\]
Then for every $n\ge 1$ there exists a solution $(\tilde{Y}^n,\tilde{Z}^n,\tilde{K}^n)$ 
to $\mathrm{\underline{R}}$\textnormal{BSDE}$^T(\xi_n,\tilde{f}_n,L)$ such that 
$\tilde{Y}^n$ is of class \textnormal{(D)}. Moreover,   $\tilde{Y}^n_t\searrow \tilde{Y}_t$ 
and $d\tilde{K}^{n+1}\ge d\tilde{K}^{n},\, n\ge 1$.
\item[(iii)] Let
\[
f_n(t,y,z)=f(t,y,z)+n(y-L_t)^--n(y-U_t)^+.
\]
Then for every $n\ge 1$ there exists a solution $(Y^n,Z^n)$ to \textnormal{BSDE}$^T(\xi,f_n,L)$ such that $Y^n$ is of class \textnormal{(D)}. 
Moreover,   $Y^n_t\to Y_t,\, t\in [0,T]$.
\end{enumerate}
 Assume additionally that $\Delta L_t\ge 0$ and $\Delta U_t\le 0$, $t\in(0,T]$. Then $Y^n\to Y$ uniformly i.e.
\[
\sup_{0\le t\le T}|Y^n_t-Y_t|\to 0,\,\,\mathbb{P}-a.s.,\quad n\to\infty,
\]
and the same type of convergence holds for $(\bar Y_n)$ and $(\tilde Y_n)$ in (i) and (ii), respectively.
\end{theorem}

\begin{proof}
By Theorem \ref{10wrzesnia2} for every $n\ge 1$ there exists a unique solution $(\bar{Y}^n,\bar{Z}^n,\bar{A}^n)$ to 
$\overline{\text{R}}\text{BSDE}^T(\xi_n,\bar{f}_n,U)$ such that $\bar{Y}^n$ is of class (D). By Proposition \ref{10wrzesnia1}, $\bar{Y}^n\le \bar{Y}^{n+1}$ and by \cite{K:BSM}, Theorem 7.4, $d\bar{A}^n\le d\bar{A}^{n+1},\, n\ge 1$. Let
\[
Y_t=\lim_{n\to\infty}\bar{Y}^n_t,\quad t\in[0,T].
\]
By Proposition \ref{10wrzesnia1}, $\bar{Y}^n\le\hat{Y}^n$, where $(\hat{Y}^n,\hat{Z}^n)$ is a unique solution to BSDE$^T(\xi,\bar{f}_n)$ 
such that $\hat Y^n$ is of class (D) (for every $n\ge 1$, the existence of the solution follows from Theorem \ref{12sierpnia1}). 
By \cite{K:BSM}, Theorem 7.4, $\hat{Y}^n\nearrow\hat{Y}$, where $(\hat{Y},\hat{Z},\hat{K})$ is a unique solution to \underline{R}BSDE$^T(\xi,f,L)$ 
such that $\hat{Y}$ is of class (D). Therefore $\bar{Y}^1\le \bar{Y}^n\le\hat{Y}$, $n\ge 1$, which implies that  $Y$ is of class (D). 
By Theorem \ref{10wrzesnia2}, for every $\varepsilon>0$ and $n\ge 1$ there exists a unique solution 
$(\bar{Y}^{n,\varepsilon},\bar{Z}^{n,\varepsilon},\bar{A}^{n,\varepsilon})$ 
to $\mathrm{\overline{R}}$\textnormal{BSDE}$^T(\xi_n,\bar{f}_{n,\varepsilon},U)$ such that $\bar{Y}^{n,\varepsilon}$ is of class (D), where
\[
\bar{f}_{n,\varepsilon}(t,y,z)=f(t,y,z)+n(y-L^{\varepsilon}_t)^-,\quad L^{\varepsilon}=L-\varepsilon.
\]
By Proposition \ref{10wrzesnia1}, 
$\bar{Y}^{n,\varepsilon}\le \bar{Y}^n$ and by \cite{K:BSM}, Theorem 7.4, 
$\bar{Y}^{n,\varepsilon}_t\nearrow Y^{\varepsilon}_t$, $t\in[0,T]$, 
where $(Y^{\varepsilon},Z^{\varepsilon},R^{\varepsilon})$ is a unique solution to 
RBSDE$^T(\xi,f,L^{\varepsilon},U)$ such that $Y^{\varepsilon}$ is of class (D). 
Hence, $L^{\varepsilon}\le Y$. This being true for each $\varepsilon>0$,
we conclude that  $L\le Y$. Obviously, $Y\le U$.

Let $\tau\in\mathcal{T}$ and $\sigma\in\Sigma_{\tau}(L,U)$.
 On every interval $[[\tau,\sigma]]$, we have 
\[
\bar{Y}^n_t=\bar{Y}^n_{\sigma}+\int^{\sigma}_t f(r,\bar{Y}^n_r,\bar{Z}^n_r)\,dr
+\int^{\sigma}_t n(\bar{Y}^n_r-L_r)^-\,dr-\int^{\sigma}_t\,d\bar{A}^n_r-\int^{\sigma}_t \bar{Z}^n_r\,dB_r,\quad t\in[\tau,\sigma].
\]
By \cite{K:BSM}, Theorem 7.4, $\bar{Y}^n\to Y^{\tau,\sigma}$, where 
$(Y^{\tau,\sigma},Z^{\tau,\sigma},R^{\tau,\sigma})$ is a unique   $\mathscr S$-solution to RBSDE$^{\tau,\sigma}(Y_{\sigma},f,L,U)$.
This shows that for any $\tau\in\mathcal{T}$ and $\sigma\in\Sigma_{\tau}(L,U)$ process $Y$ is  c\`adl\`ag  on 
$[[\tau,\sigma]]$.
We shall show that  $Y$ is c\`adl\`ag on $[0,T]$.
Let $\tau\in\mathcal{T}$ and $\sigma\in\Sigma_{\tau}(L,U)$. 
In particular, we may take $\sigma=\gamma_\tau^k$ for any $k\ge 1$.
Therefore, if $\tau<\gamma_{\tau}$, then $Y$ is right continuous in 
$\tau$, and if $\tau=\gamma_{\tau}$, then $L_{\tau}=U_{\tau}=Y_{\tau}$,  which implies, by
right continuity of $L,U$ and the fact that $L\le Y\le U$,  that  $Y$ is 
right continuous at $\tau$ again. 
Therefore, by \cite{DM}, IV.T28, $Y$ is a right continuous process on $[0,T]$. 
Let $\{\tau_l\}\subset\mathcal{T}$ be an increasing sequence and  $\tau=\sup_{l\ge 1}\tau_l$. 
Obviously, on the set 
\[
\{\omega\in\Omega: \exists_{l_\omega}\,\,\text{such that}\,\, \tau_l(\omega)=\tau(\omega),\,l\ge l_{\omega}\}\cup\{L_{\tau-}=U_{\tau-}\}
\]
the limit $\lim_{l\to\infty}Y_{\tau_l}$ exists. Now we will show that this limit exists also on the set
\[
A=\{\tau_l<\tau,\,l\ge 1\}\cap\{L_{\tau-}<U_{\tau-}\}.
\]
As we said before, $Y$ is c\`adl\`ag  on $[[\tau_l,\gamma^k_{\tau_l}]]$ for any  $k,l\ge 1$. 
Since $A\subset\{L_{\tau-}<U_{\tau-}\}$, for every $\omega\in A$ there exists $l_{\omega}$ such that
\[
[\tau_{l_{\omega}}(\omega),\tau(\omega)]\subset[\tau_{l_{\omega}}(\omega),\gamma_{\tau_{l_{\omega}}}(\omega)\}
\]
Therefore, $\lim_{l\to\infty}Y_{\tau_l}$ exists on $A$. 
Thus $\lim_{l\to\infty}Y_{\tau_l}$ exists a.s. and by \cite{DM}, IV.T28, $Y$ 
has left-side limits on $[0,T]$. This completes the proof of the existence of a  solution
to RBSDE$^T(\xi,f,L,U)$ and the proof of (i).
The proof of (ii) is analogous. As to  (iii), note that by Proposition \ref{10wrzesnia1}
\begin{equation}\label{17paz2}
\bar{Y}^n\le Y^n\le\tilde{Y}^n
\end{equation}
which gives the desired results.

Suppose now  that $\Delta L_t\ge 0$ and $\Delta U_t\le 0$, $t\in(0,T]$. 
Then, by \cite{K:SPA}, Theorem 4.7, we have that $Y$ is continuous. Therefore, by Dini's theorem
\[
\sup_{0\le t\le T}|\bar{Y}^n_t-Y_t|\to 0\quad \sup_{0\le t\le T}|\tilde{Y}^n_t-Y_t|\to 0,\quad P-\text{a.s.},\quad n\to\infty.
\]
By \eqref{17paz2} we obtain the uniform convergence of $(Y^n)$. This establishes the last assertion of the theorem.
\end{proof}

\section{Nonlinear Dynkin games}
\label{roz8}

Throughout the section we shall assume that  (H1)--(H5), (Z) are in force.

\begin{remark}
\label{rem.com}
It is worth noting  that the aforementioned conditions are imposed only to ensure that the nonlinear expectation $\mathbb E^f$
be well-defined and that the properties formulated in Proposition \ref{nonlinprop} be satisfied.  
Consequently, some of the assumptions can be dropped while additional information about the generator $f$ is given. 
In the extreme case, when $f\equiv 0$, we can dispense with all the conditions.  
\end{remark}

For  $\tau,\sigma\in\mathcal T$ we define the pay-off 
\begin{equation}\label{dyn271}
J(\tau,\sigma):=L_{\tau}\mathbf{1}_{\{\tau \le\sigma,\tau<T\}}+U_{\sigma}\mathbf{1}_{\{\sigma<\tau\}}+\xi\mathbf{1}_{\{\tau=\sigma=T\}}.
\end{equation}

\begin{definition}
Let $\theta\in\mathcal{T}$.
\begin{enumerate}
\item[(i)] Upper and lower value of the game are defined respectively as
\begin{equation}\label{13sierpnia1pr}
\begin{split}
&\overline{V}(\theta):=\mathop{\mathrm{ess\,inf}}_{\sigma\in \TT_\theta}\mathop{\mathrm{ess\,sup}}_{\tau\in \TT_\theta}\mathbb{E}^f_{\theta,\tau\wedge\sigma}J(\tau,\sigma);\\
&\underline{V}(\theta):=\mathop{\mathrm{ess\,sup}}_{\tau\in \TT_\theta}\mathop{\mathrm{ess\,inf}}_{\sigma\in \TT_\theta}\mathbb{E}^f_{\theta,\tau\wedge\sigma}J(\tau,\sigma).
\end{split}
\end{equation}
\item[(ii)] We say that nonlinear $\mathbb{E}^f$-Dynkin game with pay-off function $J$ has a value  if $\overline{V}(\theta)=\underline{V}(\theta)$
for any $\theta\in\mathcal{T}$.
\end{enumerate}
\end{definition}

For every $\theta\in\mathcal{T}$ and $\varepsilon>0$ we define the following stopping times
\begin{equation*}
\tau^{\varepsilon}_{\theta}:=\inf\{t\ge\theta,\,Y_t\le L_t+\varepsilon\}\wedge T,\quad\sigma^{\varepsilon}_{\theta}:=\inf\{t\ge\theta,\,Y_t\ge U_t-\varepsilon\}\wedge T.
\end{equation*} 
By Proposition \ref{prop7.3m}(iii) for any $\varepsilon>0$ and $\theta\in\TT$, we have
\begin{equation}
\label{eq.mm1}
[[\theta,\tau^{\varepsilon}_{\theta}]]\subset\mathscr M_{L,U}(\theta),\quad 
[[\theta,\sigma^{\varepsilon}_{\theta}]]\subset\mathscr M_{L,U}(\theta).
\end{equation}
Let $\{(\gamma_{\tau},\Lambda_{\tau}),\,\tau\in\mathcal{T}\}$ be a system  introduced in Appendix \ref{22kw1}.
Obviously $\tau^{\varepsilon}_{\theta},\sigma^{\varepsilon}_{\theta}\le\gamma_{\theta}$.

\begin{lemma}\label{12wrzesnia3}
Let $Y$ be a solution to \textnormal{RBSDE}$^T(\xi,f,L,U)$.
Then
\begin{equation}
Y_{\tau^{\varepsilon}_{\theta}}\le L_{\tau^{\varepsilon}_{\theta}}+\varepsilon,\quad\quad Y_{\sigma^{\varepsilon}_{\theta}}\ge U_{\sigma^{\varepsilon}_{\theta}}-\varepsilon.
\end{equation}
\end{lemma}
\begin{proof}
Follows immediately from the definitions of $\tau^{\varepsilon}_{\theta}$ and $\sigma^{\varepsilon}_{\theta}$.
\end{proof}

\begin{lemma}\label{12wrzesnia4}
Let $Y$ be a solution to \textnormal{RBSDE}$^T(\xi,f,L,U)$. 
Then for any $\theta\in\TT$, $Y$ is an $\mathbb{E}^f$-submartingale on $[[\theta,\tau^{\varepsilon}_{\theta}]]$ and an $\mathbb{E}^f$-supermartingale on $[[\theta,\sigma^{\varepsilon}_{\theta}]]$ (for the notion of $\mathbb{E}^f$-supermartingales  see \textnormal{Definition \ref{22kw2}}). 
\end{lemma}
\begin{proof}
We shall show that   $Y$ is an $\mathbb{E}^f$-submartingale on $[[\theta,\tau^{\varepsilon}_{\theta}]]$. 
The proof of the second assertion runs analogously. By  \eqref{eq.mm1} $(Y,\mathscr Z^{[\tau]}(Y),\mathscr V^{[\tau]}(Y+F_{Y,\mathscr Z^{[\tau]}(Y)})$ is an $\mathscr S$-solution to 
RBSDE$^{\theta,\tau^{\varepsilon}_{\theta}}(Y_{\tau^{\varepsilon}_{\theta}},f,L,U)$.
By \cite{KRz3}, Lemma 6.3, we have that $Y$ is $\mathbb{E}^f$-submartingale on $[[\theta,\tau^{\varepsilon}_{\theta}]]$.
\end{proof}

\begin{theorem}\label{19wrzesnia1}
Let  $Y$ be a solution to   \textnormal{RBSDE}$^T(\xi,f,L,U)$. Then for any $\theta\in\mathcal{T}$
\[
Y_{\theta}=\overline{V}(\theta)=\underline{V}(\theta).
\]
Moreover,  for any  $(\tau,\sigma)\in\mathcal{T}_{\theta}\times \mathcal T_\theta$
\begin{equation}\label{dyn10}
\mathbb{E}^f_{\theta,\tau\wedge\sigma^{\varepsilon}_{\theta}}J(\tau,\sigma^{\varepsilon}_{\theta})-C\varepsilon\le Y_{\theta}\le\mathbb{E}^f_{\theta,\tau^{\varepsilon}_{\theta}\wedge\sigma}J(\tau^{\varepsilon}_{\theta},\sigma)+C\varepsilon,
\end{equation}
where $C$ is a constant depending only on $\lambda,\mu,T$.
\end{theorem}
\begin{proof}
Let  $\theta\in\mathcal{T}$  and  $\varepsilon>0$. We shall prove that  
$(\tau^{\varepsilon}_{\theta},\sigma^{\varepsilon}_{\theta})$ satisfies  \eqref{dyn10}. 
By Lemma  \ref{12wrzesnia4}  $Y$ is an  $\mathbb{E}^f$-submartingale  on $[[\theta,\tau^{\varepsilon}_{\theta}]]$.
We thus have
\begin{equation}\label{dyn11}
Y_{\theta}\le\mathbb{E}^f_{\theta,\tau^{\varepsilon}_{\theta}\wedge\sigma}[Y_{\tau^{\varepsilon}_{\theta}\wedge\sigma}].
\end{equation}
By Lemma  \ref{12wrzesnia3}, $Y_{\tau^{\varepsilon}_{\theta}}\le L_{\tau^{\varepsilon}_{\theta}}+\varepsilon$. 
From this and the fact that    $Y\le U$ we have
\begin{equation*}
Y_{\tau^{\varepsilon}_{\theta}\wedge\sigma}\le (L_{\tau^{\varepsilon}_{\theta}}+\varepsilon)\mathbf{1}_{\{\tau^{\varepsilon}_{\theta}\le\sigma, \tau^{\varepsilon}_{\theta}<T\}}+U_{\sigma}\mathbf{1}_{\{\sigma<\tau^{\varepsilon}_{\theta}\}}+\xi\mathbf{1}_{\{\tau^{\varepsilon}_{\theta}=\sigma=T\}}\le J(\tau^{\varepsilon}_{\theta},\sigma)+\varepsilon.
\end{equation*}
Applying  \eqref{dyn11} and properties of the operator $\mathbb{E}^f$ (see Proposition   \ref{nonlinprop} (ii) and (v)) yields
\begin{equation}\label{dyn12}
Y_{\theta}\le\mathbb{E}^f_{\theta,\tau^{\varepsilon}_{\theta}\wedge\sigma}(J(\tau^{\varepsilon}_{\theta},\sigma)+\varepsilon)\le\mathbb{E}^f_{\theta,\tau^{\varepsilon}_{\theta}\wedge\sigma}J(\tau^{\varepsilon}_{\theta},\sigma)+C\varepsilon.
\end{equation}
By Lemma \ref{12wrzesnia4}  $Y$ is an $\mathbb{E}^f$-supermartingale  on  $[[\theta,\sigma^{\varepsilon}_{\theta}]]$. As a result
\[
Y_{\theta}\ge \mathbb{E}^f_{\theta,\tau\wedge\sigma^{\varepsilon}_{\theta}}(Y_{\tau\wedge\sigma^{\varepsilon}_{\theta}}).
\]
By Lemma  \ref{12wrzesnia3} we have  
$Y_{\sigma^{\varepsilon}_{\theta}}\ge U_{\sigma^{\varepsilon}_{\theta}}-\varepsilon$. 
Applying analogous arguments as in case of $L$ yields
$Y_{\theta}\ge\mathbb{E}^f_{\theta,\tau\wedge\sigma^{\varepsilon}_{\theta}}J(\tau,\sigma^{\varepsilon}_{\theta})-C\varepsilon$, 
 which combined with \eqref{dyn12} gives  \eqref{dyn10}. Consequently,
\begin{equation*}
Y_{\theta}\le \mathop{\mathrm{ess\,inf}}_{\sigma\ge \theta}\mathbb{E}^f_{\theta,\tau^{\varepsilon}_{\theta}\wedge\sigma}J(\tau^{\varepsilon}_{\theta},\sigma)+\varepsilon\le\mathop{\mathrm{ess\,sup}}_{\tau\ge \theta} \mathop{\mathrm{ess\,inf}}_{\sigma\ge \theta}\mathbb{E}^f_{\theta,\tau\wedge\sigma}J(\tau,\sigma)+\varepsilon,
\end{equation*}
which combined with the definition of  $\underline{V}(\theta)$ yields  $Y_{\theta}\le\underline{V}(\theta)+\varepsilon$. 
Analogous reasoning gives $\overline{V}(\theta)-\varepsilon\le Y_{\theta}$. Letting  $\varepsilon\to 0$ we find that  $\overline{V}(\theta)\le Y_{\theta}\le\underline{V}(\theta)$, which combined with the obvious inequality $\underline{V}(\theta)\le\overline{V}(\theta)$ gives $\underline{V}(\theta)=Y_{\theta}=\overline{V}_{\theta}$. 

\end{proof}

\begin{corollary}
Let $Y$ be a solution to \textnormal{RBSDE}$^T(\xi,f,L,U)$. Then for any $\theta\in \TT$, we have
\[
\begin{split}
Y_\theta&=\mathop{\mathrm{ess\,inf}}_{\sigma\in \TT_\theta^{\mathring\theta}}\mathop{\mathrm{ess\,sup}}_{\tau\in \TT^{\mathring\theta}_\theta}\mathbb{E}^f_{\theta,\tau\wedge\sigma}J(\tau,\sigma)=
\mathop{\mathrm{ess\,inf}}_{\sigma\in \TT^{\mathring\theta}_\theta}\mathop{\mathrm{ess\,sup}}_{\tau\in \TT_\theta^{\mathring\theta}}\mathbb{E}^f_{\theta,\tau\wedge\sigma}J(\tau,\sigma)\\&
=
\mathop{\mathrm{ess\,inf}}_{\sigma\in  \Sigma_{\theta}}\mathop{\mathrm{ess\,sup}}_{\tau\in \Sigma_{\theta}}\mathbb{E}^f_{\theta,\tau\wedge\sigma}J(\tau,\sigma)=
\mathop{\mathrm{ess\,inf}}_{\sigma\in  \Sigma_{\theta}}\mathop{\mathrm{ess\,sup}}_{\tau\in  \Sigma_{\theta}}\mathbb{E}^f_{\theta,\tau\wedge\sigma}J(\tau,\sigma).
\end{split}
\]
\end{corollary}

\section{Existence of saddle points}
\label{roz9}

Throughout the section we shall assume that conditions (H1)--(H5), (Z) 
are satisfied taking into account   Remark \ref{rem.com}.

\begin{lemma}[Fatou's lemma  for nonlinear expectation]
\label{list1402}
Assume that 
\[
\mathbb{E}\int^T_0|f(r,y,0)|\,dr<+\infty,\quad y\in\mathbb{R}.
\]
Let $(\tau_n)_{n\in\mathbb{N}}\subset\TT$ be a non-increasing sequence of stopping times converging 
$\mathbb{P}$-a.s. to $\tau\in\TT$. Let $(\xi_n)_{n\in\mathbb{N}}$ be a sequence of random 
variables such that $\mathbb E\sup_{n\in\mathbb{N}}|\xi_n|<\infty$ and $\xi_n$ is $\mathcal{F}_{\tau_n}$-measurable 
for $n\ge 1$. Then, for every stopping time $\alpha\le\tau$ we have 
\[
\mathbb E^f_{\alpha,\tau}(\liminf_{n\to\infty}\xi_n)\le\liminf_{n\to\infty}\mathbb E^f_{\alpha,\tau_n}(\xi_n)
\quad\mbox{and}\quad\mathbb E^f_{\alpha,\tau}(\limsup_{n\to\infty}\xi_n)\ge\limsup_{n\to\infty}\mathbb E^f_{\alpha,\tau_n}(\xi_n).
\]
\end{lemma}
\begin{proof}
We shall show the first inequality. The proof of the second one  is analogous.  For every $n$, by the monotonicity of $\mathbb E^f$ (see Proposition \ref{nonlinprop} (ii)), 
we have that $\mathbb E^f_{\alpha,\tau_n}(\inf_{p\ge n}\xi_p)\le\mathbb E^f_{\alpha,\tau_n}(\xi_n)$. Consequently 
\[
\lim_{n\to\infty}\mathbb E^f_{\alpha,\tau_n}(\inf_{p\ge n}\xi_p)\le\liminf_{n\to\infty}\mathbb E^f_{\alpha,\tau_n}(\xi_n).
\]
By the fact that $\lim_{n\to\infty}\inf_{p\ge n}\xi_p=\liminf_{n\to\infty}\xi_n$ and by the continuity of the nonlinear 
$f$-expectation (see Proposition \ref{5grudnia1}), we have that
\[
\mathbb E^f_{\alpha,\tau}(\liminf_{n\to\infty}\xi_n)=\lim_{n\to\infty}\mathbb E^f_{\alpha,\tau_n}(\inf_{p\ge n}\xi_p),
\]
which gives us the desired result. 
\end{proof}

Let $(Y,Z,R)$ be a solution to \textnormal{RBSDE}$^T(\xi,f,L,U)$. 
We shall prove that there exists a saddle point for a nonlinear Dynkin game with sufficiently regular payoffs. 
For  $\theta\in\mathcal{T}$ we define:
\begin{equation}\label{dyn14}
\tau^*_{\theta}:=\inf\{t\ge\theta,\,Y_t=L_t\}\wedge T;\quad\sigma^*_{\theta}:=\inf\{t\ge\theta,\,Y_t=U_t\}\wedge T.
\end{equation}

\begin{proposition}\label{20wrzesnia2}
Assume that  $\Delta L_t\ge 0,\, t\in (0,T]$ and  $\Delta U_t\le 0,\, t\in (0,T]$. Moreover, assume that $\mathbb{E}\int^T_0|f(r,y,0)|\,dr<+\infty$ for every $y\in\mathbb{R}$.
Let  $Y$ be a solution to   \textnormal{RBSDE}$^T(\xi,f,L,U)$ and let $\theta\in\mathcal T$.
Then $Y$  is continuous, and for any $\theta\in\TT$ it is  an  $\mathbb{E}^f$-supermartingale on  $[[\theta,\sigma^*_{\theta}]]$ and  an $\mathbb{E}^f$-submartingale on $[[\theta,\tau^*_{\theta}]]$.
\end{proposition}
\begin{proof}
Continuity of $Y$ follows from \cite{K:SPA}, Theorem 4.7.
We show that $Y$ is an $\mathbb{E}^f$-submartingale on $[[\theta,\tau^*_{\theta}]]$. 
The proof of the second assertion runs analogously.  By continuity of $Y$ and assumptions that we made
on $L$, we have $\tau^\varepsilon_\theta\nearrow \tau^*_\theta$.  Let $\nu,\delta\in \TT^{\tau^*_{\theta}}_\theta$
and $\nu\le\delta$.
By Lemma \ref{12wrzesnia4} 
\[
\mathbb{E}^f_{\nu,\nu\vee(\delta\wedge \tau^\varepsilon_{\theta})}(Y_{\nu\vee(\delta\wedge \tau^\varepsilon_{\theta})})\ge Y_\nu.
\]
By  Lemma \ref{list1402}
\begin{equation*}
\begin{split}
\mathbb{E}^f_{\nu,\delta}(Y_{\delta})=\mathbb{E}^f_{\nu,\delta}(\lim_{\varepsilon\to0}Y_{\nu\vee(\delta\wedge \tau^\varepsilon_{\theta})})\ge\limsup_{\varepsilon\to 0}\mathbb{E}^f_{\nu,\nu\vee(\delta\wedge \tau^\varepsilon_{\theta})}(Y_{\nu\vee(\delta\wedge \tau^\varepsilon_{\theta})})\ge Y_{\nu}.
\end{split}
\end{equation*} 
This completes the proof.
\end{proof}

\begin{theorem}\label{20wrzesnia3}
Assume that  $\Delta L_t\ge 0,\, t\in (0,T]$ and  $\Delta U_t\le 0,\, t\in (0,T]$.
Let  $Y$ be a solution to \textnormal{RBSDE}$^T(\xi,f,L,U)$.
Then for any  $\theta\in\mathcal{T}$ the couple  \eqref{dyn14} is a saddle point at  $\theta$ for the $\mathbb{E}^f$-Dynkin game
with the payoff \eqref{dyn271}.
\end{theorem}
\begin{proof}
Let  $\theta\in\mathcal{T}$. By Theorem  \ref{19wrzesnia1}, $Y_{\theta}=\overline{V}(\theta)=\underline{V}(\theta)$. Let  $\tau\in\mathcal T_\theta$. 
By Proposition \ref{20wrzesnia2} process $Y$ is an  $\mathbb{E}^f$-supermartingale on  $[[\theta,\tau\wedge\sigma^*_{\theta}]]$. Therefore,
\begin{equation}\label{dyn24_1}
Y_{\theta}\ge\mathbb{E}^f_{\theta,\tau\wedge\sigma^*_{\theta}}[Y_{\tau\wedge\sigma^*_{\theta}}].
\end{equation}
 Since  $Y\ge L$ and  $Y_{\sigma^*_{\theta}}=U_{\sigma^*_{\theta}}$, we also have 
 \[
 \begin{split}
 Y_{\tau\wedge\sigma^*_{\theta}}&=Y_{\tau}\mathbf{1}_{\{\tau\le\sigma^*_{\theta}, \tau<T\}}
 +Y_{\sigma^*_{\theta}}\mathbf{1}_{\{\sigma^*_{\theta}<\tau\}}+\xi\mathbf{1}_{\{\tau=\sigma^*_{\theta}=T\}}
\\&
\ge L_{\tau}\mathbf{1}_{\{\tau\le\sigma^*_{\theta}, \tau<T\}}+U_{\sigma^*_{\theta}}\mathbf{1}_{\{\sigma^*_{\theta}<\tau\}}
+\xi\mathbf{1}_{\{\tau=\sigma^*_{\theta}=T\}},
 \end{split}
 \]
 which means that $Y_{\tau\wedge\sigma^*_{\theta}}\ge J(\tau,\sigma^*_{\theta})$. 
 Using  \eqref{dyn24_1} and the fact that $\mathbb{E}^f$ is a non-decreasing operator, we deduce that   
 $Y_{\theta}\ge\mathbb{E}^f_{\theta,\tau\wedge\sigma^*_{\theta}}J(\tau,\sigma^*_{\theta})$ for any 
 $\tau\in\mathcal T_{\theta}$, in particular  
 \[
 \mathbb E^f_{\theta,\tau^*_{\theta}\wedge\sigma^*_{\theta}}J(\tau^*_{\theta},\sigma^*_{\theta})\le Y_{\theta}.
 \] 
 In the similar  way we arrive at  
 $Y_{\theta}\le\mathbb{E}^f_{\theta,\tau^*_{\theta}\wedge\sigma}J(\tau^*_\theta,\sigma)$ 
 for any $\sigma\in\mathcal{T}_{\theta}$, and so 
 \[
 Y_{\theta}\le \mathbb E^f_{\theta,\tau^*_{\theta}\wedge\sigma^*_{\theta}}J(\tau^*_{\theta},\sigma^*_{\theta}).
 \] 
 Consequently,  
 $Y_{\theta}=\mathbb E^f_{\theta,\tau^*_{\theta}\wedge\sigma^*_{\theta}}J(\tau^*_{\theta},\sigma^*_{\theta})$ 
 and  $(\tau^*_{\theta},\sigma^*_{\theta})$ is a saddle point at  $\theta$.
\end{proof}

\begin{corollary}[{\em Comparison theorem for \textnormal{RBSDEs}}]
\label{10wrzesnia1col}
Let triple $(Y^i,Z^i,R^i)$ be a solution to \textnormal{RBSDE}($\xi^i,f_i+dV^i,L^i,U^i)$,
$i=1,2$.   Assume also that  $\xi^1\le\xi^2$, $dV^1\le dV^2$, $L^1\le
L^2$, $U^1\le U^2$.
If $f_1,f_2,V^1,V^2$ satisfy \textnormal{(H1)--(H5), (Z)}, then  $Y^1_t\le Y^2_t$, $t\in[0,T]$.
\end{corollary}

\section{Strict comparison theorem and maximal saddle point}
\label{roz10}

Throughout the section we shall assume that conditions (H1)--(H5), (Z) 
are satisfied taking into account   Remark \ref{rem.com}.

\begin{enumerate}
\item[(SC)] For any  $\tau\in\TT$, $\xi^1,\xi^2\in L^1_{\mathcal{F}_{\tau}}(\Omega)$ such that  $\xi^1\ge\xi^2$,
we have 
\[
\big[\exists_{\sigma\in\TT^{\tau}}\,\,\,  \mathbb E^f_{\sigma,\tau}(\xi^1)=\mathbb E^f_{\sigma,\tau}(\xi^2)\big]\,\Rightarrow\, \xi^1=\xi^2.
\]
\end{enumerate}

\begin{proposition}
\label{stw.sc.lin2}
Let $\xi^1,\xi^2\in L^1_{\mathcal{F}_{\tau}}(\Omega)$, $\xi_1\ge\xi_2$ and $\sigma\in \TT^\tau$.
Assume  that there exist $c\ge 0$,   non-negative 
$\mathbb F$-progressively measurable processes  $\alpha,\beta$
such that $\int_\sigma^\tau(\alpha_t+\beta_t)\,dt<\infty$, and there exists a  
Borel function $\phi: [0,\infty)\to [0,\infty)$ that is bounded on bounded subsets of $[0,\infty)$,
such that  
\begin{equation}
\label{eq12.gc1}
|f(t,y_1,z)-f(t,y_2,z)|\le (\alpha_t+\beta_t\phi(|y_1-y_2|)+c|z|^2)|y_1-y_2|,
\end{equation}
$t\in [\sigma,\tau],\, y_1,y_2\in\mathbb R,\,z\in\mathbb R^d$.
If $\mathbb E^f_{\sigma,\tau}(\xi^1)=\mathbb E^f_{\sigma,\tau}(\xi^2)$, then $\xi_1=\xi_2$.
\end{proposition}
\begin{proof}
For $z\in\mathbb R^d$ we denote
$z^{(i)}:=\langle z,e_i\rangle$ and $z^{[i]}:= \sum_{k=1}^iz^{(k)}e_k$, $i=1,\dots,d$, and $z^{[0]}:=(0,\dots,0)\in\mathbb R^d$,
where $e_i,\, i=1,\dots,d$ is the standard basis in $\mathbb R^d$.
Let  $(Y^i,Z^i)$ be a  solution   to 
BSDE$^\tau(\xi_i,f)$ such that  $Y^i=\mathbb E^f_{\cdot,\tau}(\xi^i)$, $i=1,2$.
By the assumptions that we made (i.e. $\mathbb E^f_{\sigma,\tau}(\xi^1)=\mathbb E^f_{\sigma,\tau}(\xi^2)$) and the uniqueness for BSDEs,
 $Y^1=Y^2$ on $[[0,\sigma]]$. Moreover, by Proposition \ref{th.cop1},
$Y^1\le Y^2$ on $[[0,\tau]]$.  We shall show that $Y^1=Y^2$ on $[[\sigma,\tau]]$.
Set  $Y:= Y^1-Y^2$, $Z:=Z^1-Z^2$, $\xi=\xi_1-\xi_2$,
\[
g_t:=\frac{f(t,Y^1_t,Z^1_t)-f(t,Y^2_t,Z^1_t)}{Y^1_t-Y^2_t}\quad\mbox{if}\,\,  Y_t \neq 0,\text{ and } 0 \text{ otherwise},
\]
and
\[
h^i_t:=\frac{f(t,Y^2_t,Z^1_t-Z_t^{[i-1]})-f(t,Y^2_t,Z^1_t-Z_t^{[i]})}{Z_t^{(i)}} 
\quad\mbox{if}\,\, Z_t^{(i)}\neq 0 \text{ and } 0 \text{ otherwise}.
\]
By (H1) $|h^i_t|\le \lambda$,  and by (H2) and  \eqref{eq12.gc1}, 
\[
|g_t|\le (\alpha_t+\beta_t\phi(|Y^1_t-Y^2_t|)+c|Z^1_t|^2),\quad t\in [\sigma,\tau].
\]
By the assumptions that we made, and the fact that   $Y^1,Y^2$ are of class (D) and $Z^1\in \HH^s(0,T),\, s\in (0,T)$,  
one easily deduces that  
\[
\mathbb P(\int_\sigma^\tau g_s\,ds<\infty)=1.
\]
Observe that
\[
Y_t=\xi+\int_t^\tau(g_sY_s+\langle h_t,Z_t\rangle)\,ds-\int_t^\tau Z_s\,dB_s,\quad t\in [0,\tau].
\]
By the It\^o formula
\begin{equation}
\label{eq.icv}
Y_0=Y_t\Gamma_t-
\int_0^t\Gamma_sZ_s\,dB_s,\quad t\in [0,\tau],
\end{equation}
where
\[
\Gamma_t:=\exp(\int_0^t(g_s-\frac12|h_s|^2)\,ds+\int_0^th_s\,dB_s)=\exp(\int_\sigma^{\sigma\vee t}(g_s-\frac12|h_s|^2)\,ds+\int_\sigma^{\sigma\vee t}h_s\,dB_s).
\]
Let $(\beta_k)\subset \TT_\sigma$ (recall that $Z=0$ on $[[0,\sigma]]$)
be a chain on $[[0,\tau]]$ such that   $\int_0^{\cdot\wedge \tau_k}\Gamma_sZ_s\,dB_s$
is a martingale for each $k\ge 1$. Furthermore, set $\delta_k:=\inf\{t\ge \sigma:\int_\sigma^{t}g_s\,ds\ge k\}\wedge \tau$.
By the properties of $g$, $(\delta_k)$ is a chain on $[[0,\tau]]$. Set $\tau_k:=\beta_k\wedge \delta_k$.
Then, by \eqref{eq.icv}
\[
0=\mathbb E(Y_{\tau_k}\Gamma_{\tau_k}),\quad k\ge 0.
\]
By the definition of $\delta_k$, $\Gamma_{\tau_k}>0$. Hence, $\mathbb EY_{\tau_k}=0,\, k\ge1$.
By Fatou's lemma $\xi_1=\xi_2$ and hence $Y^1=Y^2$ on $[[0,\tau]]$.
\end{proof}

\begin{corollary}
If \eqref{eq12.gc1} holds with $\sigma=0$ and $\tau=T$, then (SC) holds.
\end{corollary}

\begin{theorem}
\label{stop20}
Assume that $Y$ is an $\mathbb F$-adapted process of class  \textnormal{(D)} and $\tau,\sigma\in\TT,\, \tau\le\sigma$.
\begin{enumerate}
\item[(i)] $Y$ is an  $\mathbb E^f$-supermartingale on $[[\tau,\sigma]]$ if and only if  there exist  
$K\in\mathcal{V}^{+}_{0,\mathbb{F}}(\tau,\sigma)$ and  $Z\in\HH_{\mathbb{F}}(\tau,\sigma)$ such that  
\begin{equation}
\label{stop12}
Y_t=Y_{\sigma}+\int^{\sigma}_t f(r,Y_r,Z_r)\,dr+\int^{\sigma}_t\,dK_r-\int^{\sigma}_t Z_s\,dB_r,\quad t\in[\tau,\sigma].
\end{equation}
\item[(ii)] The above  decomposition is unique, in the sense that if there exist  
$\hat K\in\mathcal{V}^{+}_{0,\mathbb F}(\tau,\sigma)$ and  $\hat Z\in\HH_{\mathbb F}(\tau,\sigma)$ such that \eqref{stop12} holds
with $(K,Z)$ replaced by $(\hat K,\hat Z)$, then $(K,Z)=(\hat K,\hat Z)$.
\item[(iii)]
If $Y$ is an $\mathbb E^f$-martingale on $[[\tau,\sigma]]$, then  \eqref{stop12} holds with $K\equiv 0$. 
\item Assume that (SC) holds and $\mathbb E^f_{\tau,\alpha}Y_\alpha=Y_\tau,\, \alpha\in \TT_\tau^\sigma$,
then $Y$ is an $\mathbb E^f$-martingale on $[[\tau,\sigma]]$.
\end{enumerate}
\end{theorem}
\begin{proof}
(i) If   $Y$ satisfies \eqref{stop12}, then, by Proposition  \ref{nonlinprop} (i), $Y$ is an $\mathbb E^f$-supermartingale on $[[\tau,\sigma]]$. Now, suppose that $Y$ is an $\mathbb E^f$-supermartingale on $[[\tau,\sigma]]$.
Let  $\theta\in\TT^{\sigma}_\tau$. Then,
\[
Y_{\theta}\ge\mathbb E^f_{\theta,\nu}(Y_{\nu})
\]
for any   $\nu\in\TT_{\theta}^\tau$.  Thus, $Y_{\theta}\ge\esssup_{\nu\in\TT_{\theta}^\sigma}\mathbb E^f_{\theta,\nu}(Y_{\nu})$.
The reverse inequality is obvious. Consequently,
\[
Y_{\theta}=\esssup_{\nu\in\TT_{\theta}^\sigma}\mathbb E^f_{\theta,\nu}(Y_{\nu}),\quad \theta\in\TT^\sigma_\tau.
\] 
By Theorem \ref{dyn20},   $Y$ is the first component of an $\mathscr S$-solution to   \underline{R}BSDE$^{\tau,\sigma}(Y_\sigma,f,Y)$.
This gives \eqref{stop12}.  

(ii) The uniqueness part follows directly from the uniqueness of the Doob--Meyer decomposition
for $\mathbb F$-semimartingales.

 (iii) Let $(V,H)$ denote a unique  solution to BSDE$^\sigma(Y_\sigma,f)$ such that $V$ is of class (D). 
By the assumption made in (iii), $V_\alpha=Y_\alpha$ for any $\alpha\in\TT_{\tau}^\sigma$ (see Remark \ref{7grudnia8}).
This gives  \eqref{stop12} with $K\equiv 0$ and $Z$ replaced by $H$. By the uniqueness, $H=Z$.

(iv) Let $\alpha\le\beta,\, \alpha,\beta\in\TT_\tau^\sigma$. By the associativity of the nonlinear expectation
and the second assumption made in (iv), we have
\[
\mathbb E^f_{\tau,\alpha}Y_\alpha=Y_\tau=\mathbb E^f_{\tau,\beta}Y_\beta
=\mathbb E^f_{\tau,\alpha}\big(\mathbb E^f_{\alpha,\beta}Y_\beta\big).
\]
Since $Y$ is an $\mathbb E^f$-supermartingale on $[[\tau,\sigma]]$, we have 
$Y_\alpha\ge \mathbb E^f_{\alpha,\beta}Y_\beta$.
Consequently, by (SC),  $Y_\alpha=\mathbb E^f_{\alpha,\beta}Y_\beta$.
\end{proof}

In what follows, in order to simplify the notation, we let for any solution  $Y$ to RBSDE$^T(\xi,f,L,U)$
and $\theta\in \TT$ (cf \eqref{eq9.5.semfff}),
\[
Z^{[\theta]}:= \mathscr Z^{[\theta]}(Y),\quad R^{[\theta]}:= \mathscr V^{[\theta]}(Y+F_{Y,Z^{[\theta]}}) \quad \text{on }\mathscr M_{L,U}(\theta)=\mathscr M_Y(\theta).
\]
We also adopt the convention that $R^{[\theta],+}_t=R^{[\theta],-}_t=+\infty,\, t\in [[\theta,T]]\setminus \mathscr M_{L,U}(\theta)$.
Let us define for any $\theta\in\TT$, 
\begin{equation}
\label{stop.hat}
\hat\tau_\theta:= \inf\{t\ge \theta: R^{[\theta],+}_t>0\},\quad \hat\sigma_\theta:= \inf\{t\ge \theta: R^{[\theta],-}_t>0\}.
\end{equation}

\begin{lemma}
\label{lm10.4}
Let $Y$ be a solution to  RBSDE$^T(\xi,f,L,U)$ and $\theta\in\TT$.
Then  $\hat{\tau}_{\theta},\hat{\sigma}_{\theta}\in\Sigma_{\theta}(L,U)$. 
\end{lemma}
\begin{proof}
Obviously $\hat{\tau}_{\theta},\hat{\sigma}_{\theta}\le\mathring\theta$. 
By the fact that $R^{[\theta],+}$ is predictable, there exists a sequence $\{\tau^k_{\theta}\}\subset\mathcal{T}_\theta$ 
such that $\tau^k_{\theta}\nearrow\hat\tau_{\theta}$ with $k\to\infty$ 
and $\tau^k_{\theta}<\hat\tau_{\theta}$, $k\ge 1$ on the set $\{\hat\tau_{\theta}>\theta\}$. 
Let $(X,H,K)$ be a solution to \underline{R}BSDE$^T(\xi,f,L)$ such that $X$ is of class (D) (see  Theorem \ref{10wrzesnia2}).
 By Proposition \ref{th.cop1}, $X_t\ge Y_t$, $t\in[0,T]$. For $l\ge 1$, we define
\[
\delta_{\theta}^l:=\inf\{t\ge\theta:\,\int_{\theta}^t f^-(r,X_r,0)\,dr\ge l\}\wedge T,
\]
and we set $\rho^{k,l}_{\theta}:=\tau^k_{\theta}\wedge\delta^l_{\theta}$. 
By \eqref{eq9.5.sem} 
$Y$ satisfies the following equation
\[
Y_t=Y_{\rho^k_{\theta}}+\int^{\rho^k_{\theta}}_t f(r,Y_r,Z^{[\theta]}_r)\,dr
-\int^{\rho^k_{\theta}}_t\,dR^{[\theta],-}_r-\int^{\rho^k_{\theta}}_t Z^{[\theta]}_r\,dB_r,\quad t\in[\theta,\rho^k_{\theta}].
\]
By \cite{K:EJP}, Lemma 3.4, we have, for $q\in(0,1)$,
\begin{equation*}
\begin{split}
&\mathbb{E}\Big(\int^{\rho^{k,l}_{\theta}}_{\theta} |Z_r|^2\,dr\Big)^{\frac{q}{2}}+\mathbb{E}(R^{[\theta],-}_{\rho^{k,l}_{\theta}})^{\frac{q}{2}}
\le C\Big(\mathbb{E}\sup_{0\le t\le T}|Y_t|^q+\mathbb{E}\Big(\int^{T}_{0}|f(r,0,0)|\,dr\Big)^q\\
&\quad+\mathbb{E}\Big(\int^{\delta^l_{\theta}}_{\theta}|f^-(r,X_r,0)|\,dr\Big)^q+\mathbb{E}\Big(\int^{T}_{0}X^+_r\,dr\Big)^q\Big)
\end{split}
\end{equation*}
for some $C>0$ independent of $k$ and $\theta$. By Fatou's lemma the above inequality 
holds with $\rho^{k,l}_{\theta}$ replaced by $\delta^l_\theta$. From that inequality  and the fact   $\{\delta^l_{\theta}\}$ is a chain 
one easily deduces   that $\int^{\hat{\tau}_{\theta}}_{\theta} |Z_r|^2\,dr<\infty$ and $R^{[\theta],-}_{\hat{\tau}_{\theta}}<\infty$ $\mathbb{P}$-a.s. 
As a result,
\[
Y_t=Y_{\hat{\tau}_{\theta}}+\int^{\hat{\tau}_{\theta}}_t f(r,Y_r,Z^{[\theta]}_r)\,dr-\int^{\hat{\tau}_{\theta}}_t\,dR^{[\theta],-}_r-\int^{\hat{\tau}_{\theta}}_t Z^{[\theta]}_r\,dB_r,\quad t\in[\theta,\hat{\tau}_{\theta}],
\]
which implies that $\hat{\tau}_{\theta}\in\Sigma_{\theta}(L,U)$. 
Similarly, we can show that $\hat{\sigma}_{\theta}\in\Sigma_{\theta}(L,U)$.
\end{proof}

\begin{proposition}\label{dyn23}
Let $Y$ be a solution to \textnormal{RBSDE}$(\xi,f,L,U)$,   $\theta\in\mathcal{T}$, and let $(\hat{\tau}_{\theta},\hat{\sigma}_{\theta})$
be the pair defined in \eqref{stop.hat}.
\begin{enumerate}
\item [(i)] If $R^{[\theta],-}$ is continuous on $\mathscr M(\theta)$, then $Y$ is an $\mathbb{E}^f$-supermartingale on $[[\theta,\hat{\sigma}_{\theta}]]$
and   $Y_{\hat{\sigma}_{\theta}}=U_{\hat{\sigma}_{\theta}}$ on  $\{\hat{\sigma}_{\theta}<T\}$.
\item[(ii)] If $R^{[\theta],+}$ is continuous on $\mathscr M(\theta)$, then $Y$ is an $\mathbb{E}^f$-submartingale on $[[\theta,\hat{\tau}_{\theta}]]$ and
 $Y_{\hat{\tau}_{\theta}}=L_{\hat{\tau}_{\theta}}$ on  $\{\hat{\tau}_{\theta}<T\}$.
\end{enumerate}
\end{proposition}
\begin{proof}
Assume that $R^{[\theta],-}$ is continuous on $\mathscr M(\theta)$. 
By the very definition of $\hat{\sigma}_{\theta}$ and the continuity assumption, $R^{[\theta],-}_{\hat{\sigma}_{\theta}}=R^{[\theta],-}_{\theta}$,  
and in turn, by the fact that $\hat{\sigma}_{\theta}\in\Sigma_{\theta}(L,U)$ (see Lemma \ref{lm10.4}), we have
\[
Y_t=Y_{\hat{\sigma}_{\theta}}+\int^{\hat{\sigma}_{\theta}}_t f(r,Y_r,Z^{[\theta]}_r)\,dr
+\int^{\hat{\sigma}_{\theta}}_t\,dR^{[\theta],+}_r-\int^{\hat{\sigma}_{\theta}}_r Z^{[\theta]}_r\, dB_r,\quad t\in[\theta,\hat{\sigma}_{\theta}].
\]
By Proposition \ref{nonlinprop}(i)  $Y$ is an $\mathbb{E}^f$-supermartingale on $[[\theta,\hat{\sigma}_{\theta}]]$.
We show that $Y_{\hat{\sigma}_{\theta}}=U_{\hat{\sigma}_{\theta}}$ on $\{\hat{\sigma}_{\theta}<T\}$. 
Suppose, by contradiction, that 
\[
\mathbb P(Y_{\hat{\sigma}_{\theta}}<U_{\hat{\sigma}_{\theta}}, \hat{\sigma}_{\theta}<T)>0,
\]
and denote by $A$ the set under the probability.
On the set $A$ for any  $a\in\mathbb{R}$ there exists  $\ve>0$ (depending on $\omega\in A$) 
such that $U_{\hat{\sigma}_{\theta}}>a+\ve$ 
and $Y_{\hat{\sigma}_{\theta}}<a-\ve$. Since $Y,U$ are c\`adl\`ag, there exists $\delta>0$ 
(also depending on $\omega\in A$) such that $U_{\hat{\sigma}_{\theta}+s}>a+\ve$ 
and $Y_{\hat{\sigma}_{\theta}+s}<a-\ve,\, s\in [0,\delta]$ on $A$.
By the definition of $\hat{\sigma}_{\theta}$ we have  $R^{[\theta],-}_{\hat{\sigma}_{\theta}+\delta}>R^{[\theta],-}_{\hat{\sigma}_{\theta}}$. 
Therefore, on the set $A$ we obtain the following inequality
\[
\int^{\hat{\sigma}_{\theta}+\delta}_{\hat{\sigma}_{\theta}}(U_{r-}-Y_{r-})\,dR^{[\theta],-}_r>2\ve(R^{[\theta],-}_{\hat{\sigma}_{\theta}+\delta}-R^{-,*}_{\hat{\sigma}_{\theta}})>0,
\]
which contradicts the minimality condition. This completes the proof of (i).
The proof of (ii) is analogous. 
\end{proof}

\rm{

\begin{corollary}\label{luty06}
Let $(\tau^*_{\theta},\sigma^*_{\theta})$
be the pair defined in \eqref{dyn14}.
Under the notation of  Proposition \ref{dyn23} the following implication holds:   if $R^{[\theta],-}$ (resp. $R^{[\theta],+}$) is continuous on $\mathscr M(\theta)$,
 then $\sigma^*_{\theta}\le\hat{\sigma}_{\theta}$ (resp. $\tau^*_{\theta}\le\hat{\tau}_{\theta}$).
\end{corollary}

\begin{proposition}\label{dyn26}
Let $Y$ be a solution to \textnormal{RBSDE}$^T(\xi,f,L,U)$ and let $\theta\in\mathcal{T}$. Asume that $\Delta L_t\ge 0$ and $\Delta U_t\le 0$, $t\in(0,T]$. Then, for every $\theta\in\mathcal{T}$, the pair $(\hat{\tau}_{\theta},\hat{\sigma}_{\theta})$ defined in \eqref{stop.hat} is a saddle point at $\theta$
for the $\mathbb{E}^f$-Dynkin game with payoff \eqref{dyn271}.
\end{proposition}
\begin{proof}
Let $\theta\in\mathcal{T}$. By Proposition \ref{20wrzesnia2}, 
$Y$ is continuous, therefore $R^{[\theta],+}$, $R^{[\theta],-}$ are continuous. Let $\tau\in\mathscr{M}_{L,U}(\theta)$.
By Proposition \ref{dyn23}(i), $Y$ is an $\mathbb{E}^f$-supermartingale on $[[\theta,\tau\wedge\hat{\sigma}_{\theta}]]$. Therefore,
\begin{equation}\label{dyn24}
Y_{\theta}\ge\mathbb{E}^f_{\theta,\tau\wedge\hat{\sigma}_{\theta}}[Y_{\tau\wedge\hat{\sigma}_{\theta}}].
\end{equation}
By the fact that $Y\ge L$ and $Y_{\hat{\sigma}_{\theta}}=U_{\hat{\sigma}_{\theta}}$ on  $\{\hat{\sigma}_{\theta}<T\}$, 
(see Proposition \ref{dyn23}) we have that for any $\tau\in \TT_\theta$,
\[
\begin{split}
Y_{\tau\wedge\hat{\sigma}_{\theta}}&
=Y_{\tau}\mathbf{1}_{\{\tau\le\hat{\sigma}_{\theta},\tau<T\}}+Y_{\hat{\sigma}_{\theta}}\mathbf{1}_{\{\hat{\sigma}_{\theta}<\tau\}}
+\xi\mathbf{1}_{\{\tau=\hat{\sigma}_{\theta}=T\}}
\\&\ge L_{\tau}\mathbf{1}_{\{\tau\le\hat{\sigma}_{\theta},\tau<T\}}
+U_{\hat{\sigma}_{\theta}}\mathbf{1}_{\{\hat{\sigma}_{\theta}<\tau\}}+\xi\mathbf{1}_{\{\tau=\hat{\sigma}_{\theta}=T\}}=J(\tau,\hat{\sigma}_{\theta}).
\end{split}
\]
By \eqref{dyn24} and the monotonicity of $\mathbb{E}^f$ (see Proposition \ref{nonlinprop}(ii)), we have  
$Y_{\theta}\ge\mathbb{E}^f_{\theta,\tau\wedge\hat{\sigma}_{\theta}}J(\tau,\hat{\sigma}_{\theta})$ 
for any $\tau\in\TT_{\theta}$, so in particular 
$\mathbb E^f_{\theta,\hat{\tau}_{\theta}\wedge\hat{\sigma}_{\theta}}J(\hat{\tau}_{\theta},\hat{\sigma}_{\theta})\le Y_{\theta}$. 
Similarly, we can show that for every $\sigma\in\mathcal{T}_{\theta}$ the inequality 
$Y_{\theta}\le\mathbb{E}^f_{\theta,\hat{\tau}_{\theta}\wedge\sigma}J(\hat{\tau}_\theta,\sigma)$
is satisfied, thus $Y_{\theta}\le \mathbb E^f_{\theta,\hat{\tau}_{\theta}\wedge\hat{\sigma}_{\theta}}J(\hat{\tau}_{\theta},\hat{\sigma}_{\theta})$.
This implies that $Y_{\theta}=\mathbb E^f_{\theta,\hat{\tau}_{\theta}\wedge\hat{\sigma}_{\theta}}J(\hat{\tau}_{\theta},\hat{\sigma}_{\theta})$ 
and $(\hat{\tau}_{\theta},\hat{\sigma}_{\theta})$ is a saddle point at $\theta$.
\end{proof}

}

\begin{theorem}
\label{th.mal}
Suppose that $R^{[\theta]}$ is continuous on $\mathscr M_{L,U}(\theta)$. Let $(\tau_\theta,\sigma_\theta)$
be a saddle point for the $\mathbb{E}^f$-Dynkin game with payoff \eqref{dyn271}. Then $Y$ is a martingale on $[[\theta,\tau_\theta\wedge\sigma_\theta]]$,
$\tau^*_\theta\le \tau_\theta$, $\sigma_\theta^*\le\sigma_\theta$,
and $\tau_\theta\wedge\sigma_\theta\le\hat\tau_\theta\wedge \hat\sigma_\theta$. Moreover,
\begin{equation}
\label{eq.vl}
Y_{\tau_\theta\wedge\sigma_\theta}=L_{\tau_\theta}\mathbf{1}_{\{\tau_\theta \le\sigma_\theta,\tau_\theta<T\}}
+U_{\sigma_\theta}\mathbf{1}_{\{\sigma_\theta<\tau_\theta\}}+\xi\mathbf{1}_{\{\tau_\theta=\sigma_\theta=T\}}.
\end{equation}
\end{theorem}
\begin{proof}
Let $\delta_\theta:=\tau_\theta\wedge\sigma_\theta$. 
Using the fact that $Y$ represents the value of the game and $(\tau_\theta,\sigma_\theta)$
is a saddle point for the game, we deduce that for any $\alpha\in \TT_\theta^{\delta_\theta}$
\begin{equation*}
\begin{split}
Y_\theta&=\mathbb E^f_{\theta,\delta_\theta}J(\tau_\theta,\sigma_\theta)=\mathop{\mathrm{ess\,sup}}_{\tau\in \TT_{\theta}}\mathbb E^f_{\theta,\tau\wedge\sigma_\theta}J(\tau,\sigma_\theta)
=
\mathop{\mathrm{ess\,sup}}_{\tau\in \TT_{\alpha}}\mathbb{E}^f_{\theta,\tau\wedge\sigma_\theta}J(\tau,\sigma_\theta)
\\&\ge
\mathop{\mathrm{ess\,inf}}_{\sigma\in \TT_\alpha}\mathop{\mathrm{ess\,sup}}_{\tau\in \TT_{\alpha}}\mathbb{E}^f_{\theta,\tau\wedge\sigma}J(\tau,\sigma)
=
\mathbb E^f_{\theta,\alpha}\Big(\mathop{\mathrm{ess\,inf}}_{\sigma\in \TT_\alpha}\mathop{\mathrm{ess\,sup}}_{\tau\in \TT_{\alpha}}\mathbb{E}^f_{\alpha,\tau\wedge\sigma}J(\tau,\sigma)\Big)=\mathbb E^f_{\theta,\alpha}Y_\alpha,
\end{split}
\end{equation*}
and, analogously, 
\begin{equation*}
\begin{split}
Y_\theta&=\mathbb E^f_{\theta,\delta_\theta}J(\tau_\theta,\sigma_\theta)
=\mathop{\mathrm{ess\,inf}}_{\sigma\in \TT_{\theta}}\mathbb E^f_{\theta,\tau_\theta\wedge\sigma}J(\tau_\theta,\sigma)
=
\mathop{\mathrm{ess\,inf}}_{\sigma\in \TT_{\alpha}}\mathbb{E}^f_{\theta,\tau_\theta\wedge\sigma}J(\tau_\theta,\sigma)
\\&\le
\mathop{\mathrm{ess\,sup}}_{\tau\in \TT_{\alpha}}\mathop{\mathrm{ess\,inf}}_{\sigma\in \TT_\alpha}\mathbb{E}^f_{\theta,\tau\wedge\sigma}J(\tau,\sigma)
=
\mathbb E^f_{\theta,\alpha}\Big(\mathop{\mathrm{ess\,sup}}_{\tau\in \TT_{\alpha}}\mathop{\mathrm{ess\,inf}}_{\sigma\in \TT_\alpha}\mathbb{E}^f_{\alpha,\tau\wedge\sigma}J(\tau,\sigma)\Big)=\mathbb E^f_{\theta,\alpha}Y_\alpha.
\end{split}
\end{equation*}
Therefore,
\begin{equation}
\label{eq12.id}
Y_\theta=\mathbb E^f_{\theta,\alpha}Y_\alpha,\quad \alpha\in \TT_\theta^{\delta_\theta}.
\end{equation}
From this, in particular, we obtain 
\[
\mathbb E^f_{\theta,\delta_\theta}Y_{\delta_\theta}=Y_\theta=
\mathbb E^f_{\theta,\delta_\theta} (L_{\tau_\theta}\mathbf{1}_{\{\tau_\theta \le\sigma_\theta,\tau_\theta<T\}}+U_{\sigma_\theta}\mathbf{1}_{\{\sigma_\theta<\tau_\theta\}}+\xi\mathbf{1}_{\{\tau_\theta=\sigma_\theta=T\}}).
\]
We know that $\hat\tau_\theta, \hat\sigma_\theta\in \Sigma_{\theta}(L,U)$. 
 Thus,
\begin{equation}
\label{eq12.1}
Y_t=Y_{\hat\tau_\theta\wedge \delta_\theta}+\int^{\hat\tau_\theta\wedge \delta_\theta}_t f(r,Y_r,Z_r)\,dr-\int^{\hat\tau_\theta\wedge \delta_\theta}_t\,dR^{[\theta],-}_r-\int^{\hat\tau_\theta\wedge \delta_\theta}_t Z_s\,dB_r,\quad t\in[\theta,\hat\tau_\theta\wedge \delta_\theta],
\end{equation}
and 
\begin{equation}
\label{eq12.2}
Y_t=Y_{\hat\sigma_\theta\wedge \delta_\theta}+\int^{\hat\sigma_\theta\wedge \delta_\theta}_t f(r,Y_r,Z_r)\,dr+\int^{\hat\sigma_\theta\wedge \delta_\theta}_t\,dR^{[\theta],+}_r-\int^{\hat\sigma_\theta\wedge \delta_\theta}_t Z_s\,dB_r,\quad t\in[\theta,\hat\sigma_\theta\wedge \delta_\theta].
\end{equation}
Furthermore,
\[
Y_t=Y_{\hat\sigma_\theta\wedge \hat\tau_\theta}+\int^{\hat\sigma_\theta\wedge \hat\tau_\theta}_t f(r,Y_r,Z_r)\,dr
-\int^{\hat\sigma_\theta\wedge \hat\tau_\theta}_t Z_s\,dB_r,\quad t\in[\theta,\hat\sigma_\theta\wedge \hat\tau_\theta].
\]
From the last equation, we have that $Y$ is an $\mathbb E^f$-martingale on 
$[[\theta,\hat\sigma_\theta\wedge \hat\tau_\theta]]$.
From   \eqref{eq12.id} -- \eqref{eq12.2},  by  Theorem \ref{stop20}(iv), we have 
$R^{[\theta],-}\equiv 0$ on $[[\theta,\hat\tau_\theta\wedge \delta_\theta]]$
and $R^{[\theta],+}\equiv 0$ on $[[\theta,\hat\sigma_\theta\wedge \delta_\theta]]$. From these equations, we infer that
$R^{[\theta]}\equiv 0$ on $[[\theta,\delta_\theta]]$ (since on the interval $[[\theta,\delta_\theta]]$ process $R^{[\theta],-}$
does not increase until $R^{[\theta],+}$ increases and vice versa). As a result $\delta_\theta\le\hat\tau_\theta\wedge  \hat\sigma_\theta$.
\end{proof}

\begin{appendix}

\section{Nonlinear expectation}
\label{rozA}

Throughout this section, we assume that (H1)-(H4) are in force.
In what follows $p$ always denotes a real number greater than $1$.
Let $\nu,\zeta\in\mathcal{T}$, $\nu\le\zeta$. We define under conditions (H1)-(H5), (Z) the operator
\[
\mathbb{E}^{f,(1)}_{\nu,\zeta}:L^1(\mathcal{F}_{\zeta})\to L^1(\mathcal{F}_{\nu}),
\]
by letting $\mathbb{E}^{f,(1)}_{\nu,\zeta}(\xi):=Y_{\nu}$,
where $(Y,Z)$ is a  solution to BSDE$^{\zeta}(\xi,f)$ such that $Y$ is of class (D).
We also define under conditions (H1)-(H4), (H5$_p$) the operator
\[
\mathbb{E}^{f,(p)}_{\nu,\zeta}:L^p(\mathcal{F}_{\zeta})\to L^p(\mathcal{F}_{\nu}),
\]
by letting $\mathbb{E}^{f,(p)}_{\nu,\zeta}(\xi):=Y_{\nu}$,
where $(Y,Z)$ is a  solution to BSDE$^{\zeta}(\xi,f)$ such that $Y\in \mathcal S^p_{\mathbb{F}}(0,T)$.

By Theorems \ref{12sierpnia1}, \ref{12sierpnia1p}  both operators are well defined. We let 
for any $\xi\in L^1(\FF_\zeta)$
\[
\mathbb{E}^f_{\nu,\zeta}(\xi)=
 \left\{
\begin{array}{ll} \mathbb{E}^{f,(1)}_{\nu,\zeta}, &\xi\in L^1(\FF_{\zeta})\setminus \bigcup_{p>1}L^p(\FF_{\zeta}),
\smallskip\\ \mathbb{E}^{f,(p)}_{\nu,\zeta}, &\xi \in L^p(\FF_{\zeta}).
\end{array}
\right.
\]

\begin{definition}\label{22kw2}
We say that a process $X$ of class (D) is an $\mathbb{E}^f$-supermartingale (resp. $\mathbb{E}^f$-submartingale) 
on $[[\nu,\zeta]]$, if $\mathbb{E}^f_{\sigma,\tau}(X_{\tau})\le X_{\sigma}$ 
(resp. $\mathbb{E}^f_{\sigma,\tau}(X_{\tau})\ge X_{\sigma}$) for every $\sigma,\tau\in\mathcal{T}_{\nu,\zeta}$, $\sigma\le\tau$. 
$X$ is an $\mathbb{E}^f$-martingale on $[[\nu,\zeta]]$, if $X$ is at the same time an 
$\mathbb{E}^f$-supermartingale and an $\mathbb{E}^f$-submartingale on $[[\nu,\zeta]]$.
\end{definition}

\begin{remark}\label{7grudnia8}
Note that the process $Y$ of class \textnormal{(D)} is an $\mathbb{E}^f$-martingale on $[[\nu,\zeta]]$ 
if and only if $Y$ is indistinguishable from the first component of the solution to BSDE$^{\nu,\zeta}(Y^\zeta,f)$ on $[[\nu,\zeta]]$.
Thus, in order to prove that   $Y$ is an $\mathbb{E}^f$-martingale on $[[\nu,\zeta]]$, it suffices to show 
that $Y_{\sigma}=\mathbb{E}^f_{\sigma,\zeta}(Y_{\zeta})$, for any  $\sigma\in\mathcal{T}_{\nu,\zeta}$.
\end{remark}

The following result have been proven in \cite{KRz3}, Proposition 2.

\begin{proposition}\label{nonlinprop}

Let  $\nu,\zeta\in\mathcal{T}$, $\nu\le\zeta$.
\begin{enumerate}
\item[(i)] Let $\xi\in L^1(\mathcal{F}_{\zeta})$ and let $V$ be an $\mathbb{F}$-adapted, 
finite variation process such that $V_{\nu}=0$. Let $(X,H)$ be a solution to \textnormal{BSDE}$^{\nu,\zeta}(\xi,f+dV)$ 
such that $X$ is of class \textnormal{(D)}. If $V$ (resp. $-V$) is increasing, then $X$ 
is an $\mathbb{E}^f$-supermartingale (resp. $\mathbb{E}^f$-submartingale) on $[[\nu,\zeta]]$.
\item[(ii)] If $\xi_1,\xi_2\in L^1(\mathcal{F}_{\zeta})$, $\xi_1\le\xi_2$ and for a.e. $t\in[0,T]$ we have that $f^1(t,y,z)\le
f^2(t,y,z)$ for $y\in\mathbb{R}$, $z\in\mathbb{R}^d$, then $\mathbb{E}^{f^1}_{\nu,\zeta}(\xi_1)\le\mathbb{E}^{f^2}_{\nu,\zeta}(\xi_2)$.
\item[(iii)] Let $\xi\in L^1(\mathcal{F}_{\zeta})$. For every $A\in\mathcal{F}_{\nu}$,
\[
\mathbf{1}_A\mathbb{E}^f_{\nu,\zeta}(\xi)=\mathbb{E}^{f_A}_{\nu,\zeta}(\mathbf{1}_A\xi),
\]
where $f_A(t,y,z)=f(t,y,z)\mathbf{1}_A\mathbf{1}_{\{t\ge\nu\}}$.
\item[(iv)] Let $\xi\in L^1(\mathcal{F}_{\zeta})$. For every $\gamma\in\mathcal{T}$ such that $\gamma\ge\zeta$,
\[
\mathbb{E}^f_{\nu,\zeta}(\xi)=\mathbb{E}^{f^{\zeta}}_{\nu,\gamma}(\xi),
\]
where $f^{\zeta}(t,y,z)=f(t,y,z)\mathbf{1}_{\{t\le\zeta\}}$.

\item[(v)] 
Let $\nu,\zeta_1,\zeta_2\in\mathcal{T}$, $\nu\le\zeta_1\le\zeta_2$, and 
 $\xi_1\in L^1(\mathcal{F}_{\zeta_1})$,  $\xi_2\in L^1(\mathcal{F}_{\zeta_2})$. 
Furthermore, let $(Y^1,Z^1)$ be a solution to \textnormal{BSDE}$^{\nu,\zeta_2}(\xi_2,f_1^{\zeta_1})$, 
where $f_1^{\zeta_1}(t,y,z)=f_1(t,y,z)\mathbf{1}_{\{t\le\zeta\}}$, and $(Y^2,Z^2)$ 
be a solution to \textnormal{BSDE}$^{\nu,\zeta_2}(\xi_2,f_2)$,
such that $Y^1, Y^2$ are of class (D). If $\mathbb{E}\sup_{\nu\le t\le\zeta_2}|Y^1_t-Y^2_t|^2<\infty$, then
\begin{equation*}
\begin{split}
&|\mathbb{E}^{f_1}_{\nu,\zeta_1}(\xi_1)-\mathbb{E}^{f_2}_{\nu,\zeta_2}(\xi_2)|^2\le C \mathbb{E}\Big(\int^{\zeta_1}_\nu|Y^1_r-Y^2_r||f_1-f_2|(r,Y^2_r,Z^2_r)\,dr\\
&\quad+|\xi_1-\xi_2|^2+\int^{\zeta_2}_{\zeta_1}|Y^1_r-Y^2_r||f_2(r,Y^2_r,Z^2_r)|\,dr|\mathcal{F}_{\nu}\Big),
\end{split}
\end{equation*}
for some $C$ depending only on $\lambda,\mu, T$.
\end{enumerate}
\end{proposition}

\begin{proposition}\label{5grudnia1}
Assume that $\mathbb{E}\int^T_0|f(r,y,0)|\,dr<+\infty$ for every $y\in\mathbb{R}$. 
Let $\beta\in\TT$ and let $\{\beta_n\}_{n\ge 1}\subset\TT$ be a sequence of stopping times such 
that $\beta_n\searrow\beta$, $n\to\infty$. Assume that $\xi\in L^1(\Omega,\mathcal{F}_{\beta},P)$. 
Moreover, let $\{\xi_n\}_{n\in\mathbb{N}}$ be a sequence of random variables such that 
$\xi_n\in L^1(\Omega,\mathcal{F}_{\beta_n},P)$, $n\in\mathbb{N}$, and $\xi_n\searrow\xi$, $n\to\infty$. 
Then $||\mathbb{E}^f_{\cdot,\beta_n}(\xi_n)-\mathbb{E}^f_{\cdot,\beta}(\xi)||_{\mathcal{D}^1(0,\beta)}\to 0$, $n\to\infty$.
\end{proposition}
\begin{proof}
Let $(Y,Z)$ be a solution to BSDE$^{\alpha,\beta_1}(\xi,f^{\beta})$ such that $Y$ is of class (D), 
where $f^{\beta}(t,y,z)=f(t,y,z)\mathbf{1}_{t\le\beta}$,  and for all $n\in\mathbb{N}$
let $(Y^n,Z^n)$ be a solution to BSDE$^{\alpha,\beta_1}(\xi_n,f^{\beta_n})$ such that $Y^n$ is of class (D), 
where $f^{\beta_n}(t,y,z)=f(t,y,z)\mathbf{1}_{t\le\beta_n}$. The existence of the solutions follows from Theorem \ref{12sierpnia1}. 
Notice  that $Y_t=\xi$ 
and $Z_t=0$ for $t\ge\beta$. Obviously $\mathbb E^f_{\alpha,\beta}(\xi)=Y_{\alpha}$ and 
$Y^n_{\alpha}=\mathbb E^{f^{\beta_n}}_{\alpha,\beta_n}(\xi_n)$, $\alpha\le\beta$, $n\in\mathbb{N}$. 
The rest of the proof we divide into two steps.

\textbf{Step 1.} We assume additionally that $\xi$ is bounded i.e. $|\xi|\le M$ for some $M>0$. Then by Theorem \ref{th.2} we have that
\[
\begin{split}
&\|Y^n-Y\|_{\mathcal D^1(0,\beta_1)}\le \mathcal{C}\psi_3 \Big(\mathbb{E}|\xi_n-\xi|+\mathbb{E}\int^{\beta_n}_{\beta}|f(r,\xi,0)|\,dr\Big)\\
&\quad\le \mathcal{C}\psi_3 \Big(\mathbb{E}|\xi_n-\xi|+\mathbb{E}\int^{\beta_n}_{\beta}|f(r,-M,0)|+|f(r,M,0)|\,dr\Big).
\end{split}
\]
By (H8) and by the Lebesgue dominated convergence theorem we have the result.

\textbf{Step 2.} Let us consider $\Pi_k(x)=(x\vee(-k))\wedge k$, $k\ge 1$, and let $(Y^k,Z^k)$ be a solution to 
BSDE$^{\beta_1}(\Pi_k(\xi),f^{\beta})$ such that $Y^k$ is of class (D). 
The existence of the solutions follows from Theorem \ref{12sierpnia1}. 
By Theorem \ref{th.2} $\|Y^k-Y\|_{\mathcal D^1(0,\beta_1)}\to 0$ when $k\to\infty$. 
Similarly, let $(Y^{n,k},Z^{n,k})$ be a solution to BSDE$^{\beta_1}(\Pi_k(\xi_n),f^{\beta_n})$
such that $Y^{n,k}$ is of class (D). 
The existence of the solutions follows from Theorem \ref{12sierpnia1}. 
By \textbf{Step 1.} 
we have  $\|Y^{n,k}-Y^k\|_{\mathcal D^1(0,\beta_1)}\to 0$ when $n\to\infty$. By Theorem \ref{th.2},
\[
\begin{split}
\|Y^{n,k}-Y^n\|_{\mathcal D^1(0,\beta_1)}\le \mathcal{C}\psi_3 \Big(\mathbb{E}|\Pi_k(\xi_n)-\xi_n|\Big)\le 
\mathcal{C}\psi_3 \Big(\mathbb{E}(|\xi_1|+|\xi|)\mathbf{1}_{\{|\xi_1|+|\xi|>k\}}\Big)=:\delta_k,
\end{split}
\]
and by the Lebesgue dominated convergence theorem $\delta_k\to 0$, $k\to\infty$. Therefore, 
\[
\|Y^n-Y\|_{\mathcal D^1(0,\beta_1)}\le \mathcal{C}\big(\delta_k+\|Y^{n,k}-Y^k\|_{\mathcal D^1(0,\beta_1)}+\|Y^k-Y\|_{\mathcal D^1(0,\beta_1)}\big).
\]
Consequently,
\[
\limsup_{n\to\infty}\|Y^n-Y\|_{\mathcal D^1(0,\beta_1)}\le \mathcal{C} \big(\|Y^k-Y\|_{\mathcal D^1(0,\beta)}+\delta_k\big),\quad k\ge 1.
\]
Passing to the limit with $k\to\infty$ yields  the result.
\end{proof}

\section{Non-semimartingale solutions to RBSDEs with   driver independent of  control variable}\label{22kw1}
\label{rozB}

In \cite{K:SPA} the first author of the present paper has introduced 
the notion of solutions of RBSDE$^T(\xi,f,L,U)$ on a general filtration satisfying usual conditions
in case $f$ depends only on $y$-variable (only standing assumption (B) imposed on $L,U$). 
We shall recall this notion but adapted to the framework  of the present paper.

For  $\tau\in\TT$, we let
\[
\dot\gamma_\tau=\inf\{\tau<t\le T: L_{t-}=U_{t-}\}\wedge\inf\{\tau\le t\le T:L_t=U_t\},
\]
and then 
\begin{equation}
\label{eq2.1}
\gamma_\tau=\dot\gamma_\tau\wedge T,\qquad \Lambda_\tau=\{L_{\gamma_\tau-}=U_{\gamma_\tau-}\}\cap\{\tau<\gamma_\tau\}.
\end{equation}
Observe that $\Lambda_\tau\in\FF_{\gamma_\tau-}$. 
Let $(\dot\delta^{k}_\tau)$, with   $\dot\delta^{k}_\tau\ge\tau$, be an announcing sequence for $(\gamma_\tau)_{\Lambda_\tau}$ and let $\delta^{k}_\tau=\dot\delta^{k}_\tau\wedge T$. We put
\begin{equation}
\label{eq2.3cdf}
\gamma^{k}_\tau=\delta_\tau^{k}\wedge\gamma_\tau.
\end{equation}

\begin{definition}
\label{df.main}
We say that an  $\mathbb F$-adapted c\`adl\`ag process $Y$  is a solution 
to the problem RBSDE$^T(\xi,f,L,U)$  if
\begin{enumerate}
\item[(a)] $L_t\le Y_t\le U_t,\, t\in [0,T]$, $Y_T=\xi$,
\item[(b)] for any $\tau\in\mathcal T$ and any $k\ge 1$ process $Y$
is a semimartingale on $[[\tau,\gamma^k_\tau]]$,
\item[(c)] for every $\tau\in\TT,\, k\ge 1$,
\[
\int_{[\tau, {\gamma^k_\tau}]}(Y_{r-}-L_{r-})\,d\Gamma^{v,+}_r(\tau)
=\int_{[\tau,{\gamma^k_\tau}]}(U_{r-}-Y_{r-})\,d\Gamma^{v,-}_r(\tau)=0,
\]
where $\Gamma_t:= -Y_t+Y_0-\int_0^t f(r,Y_r)\,dr,\, t\in [0,T]$, and $\Gamma^v(\tau)$
is the finite variation c\`adl\`ag process coming from the Doob--Meyer decomposition
of $\Gamma$ on $[[\tau,\gamma_\tau^k]]$.
\end{enumerate}
\end{definition}

\end{appendix}

\end{document}